\documentclass[11pt,twoside]{amsart}
\usepackage{amsmath,hyperref,amstext,amssymb,amsthm,mathrsfs,eurosym,dsfont}
\usepackage{xcolor}
\usepackage{enumerate}   
\usepackage{enumitem}
\usepackage{tikz}
\usepackage{booktabs}
\usepackage[a4paper,hcentering,vcentering,margin = 1in]{geometry}
\usepackage{comment}

\newtheorem{theorem}{Theorem}[section]
\newtheorem{lemma}[theorem]{Lemma}

\theoremstyle{definition}

\newtheorem{assumption}[theorem]{Assumption}

\newtheorem{remark}[theorem]{Remark}

\numberwithin{equation}{section}
\allowdisplaybreaks

\newcommand{\w}{\varpi}
\newcommand{\wmf}{\varpi^{\text{mf}}}
\newcommand{\E}{\mathbb{E}}
\newcommand{\F}{\mathbb{F}}
\newcommand{\G}{\mathbb{G}}
\newcommand{\R}{\mathbb{R}}

\newcommand{\ItF}{\mathcal{F}}
\newcommand{\prob}{\mathbb{P}}
\newcommand{\setup}{(\Omega,\mathcal{F},\mathbb{P},\mathbb{F})}

\newcommand{\cadlag}{c\`{a}dl\`{a}g }

\renewcommand{\bar}{\overline}
\renewcommand{\hat}{\widehat}
\renewcommand{\tilde}{\widetilde}

\newcommand{\avg}{\frac{1}{N}\sum_{j = 1}^N}

\newcommand{\X}{\mathcal{X}}

\title[Mean-field model for pollution abatement through ETS system]{Mean-field model for pollution abatement via cap and trade mechanism}

\author[O. Bonesini]{Ofelia Bonesini}\address{London School of Economics and Political Science, LSE, London}\email{o.bonesini@lse.ac.uk}

\author[G. Lanaro]{Giacomo Lanaro}\address{Prometeia SPA, Milan, Italy}\email{giacomo.lanaro94@gmail.com}

\date{\today}

\begin{document}

\begin{abstract}

We consider a mean-field model of competitive firms operating under an AK production technology, where output is proportional to capital and production generates emissions. In this setting, we introduce a regulator whose objective is the reduction of cumulative emissions, in the spirit of Emission Trading Systems (ETS), where firms must hold and trade permits to cover their emissions in a regulated market.
The regulator acts as a central planner and controls the supply of permits, balancing emission reduction and aggregate output. Permits are allocated through a dynamic auction mechanism that adjusts supply in real time to achieve both market efficiency and the regulator’s long-term goals. The regulator and the firms interact through the endogenous permit price, which is determined by the regulator’s policy and affects firms’ optimal strategies via the market clearing condition.
The regulator’s optimal policy is derived within a Mean-Field Control (MFC) framework. Exploiting the linear-quadratic structure and the presence of common noise, we characterise the equilibrium via a system of coupled FBSDEs and associated Riccati equations.
Our results provide a tractable characterisation of optimal permit allocation policies in large economies and offer insights into the design of efficient emission trading mechanisms.
\end{abstract}

\maketitle
\section{Introduction}
\label{section: introduction}
Addressing the environmental impact of industrial production has become a cornerstone of modern economic policy. Firms’ emissions contribute to a range of negative externalities, from environmental degradation to public health concerns, underscoring the need for effective regulatory interventions. Among the policy instruments available, emission trading systems (ETS), such as the European Union Emission Trading Scheme (EU-ETS), provide a market-based approach for controlling pollution. By setting a cap on total emissions and allowing firms to trade permits, ETS frameworks incentivise firms to balance production decisions with environmental considerations while maintaining economic efficiency.

ETS have emerged as one of the most widely adopted market-based instruments for reducing greenhouse gas (GHG) emissions. Originating from the idea of creating a tradable right to pollute under an overall emissions cap, these systems allow policymakers to achieve environmental targets while preserving flexibility for regulated entities. Over the past two decades, the use of ETS has expanded far beyond the European Union. According to the International Carbon Action Partnership (ICAP), more than thirty jurisdictions worldwide, including national systems in countries such as China, South Korea, and the United Kingdom, as well as regional initiatives such as the Western Climate Initiative (WCI) in North America, have implemented or are developing their own trading schemes. Together, these systems now cover around one-sixth of global greenhouse gas emissions.\footnote{See the ICAP Status Report (2022), \emph{Emissions Trading Worldwide}, available at \href{https://icapcarbonaction.com/en/publications/emissions-trading-worldwide-2022-icap-status-report}{icapcarbonaction.com}.}

Despite differences in design, such as sectoral coverage, allocation methods, or price stabilisation mechanisms, all ETS share a common economic rationale: by assigning a price to carbon, they internalise the externality associated with emissions and promote cost-effective abatement. This global proliferation highlights both the flexibility and the policy relevance of ETS as a central component of climate strategies. 

While a growing literature applies mean-field games to energy transition, including \cite{chan2017fracking} on fracking and renewables, \cite{carmona2022mean} on carbon regulation with deep learning methods, and \cite{dumitrescu2024energy} on optimal stopping under scenario uncertainty, most existing works treat the regulator as a passive actor who sets a fixed policy or price, leaving firms' responses as the central question. The regulator's own optimisation problem remains largely unexplored in the mean-field setting, and it is precisely this gap that our work addresses.

In particular, the mathematical study of permit markets for carbon reduction dates back to the work~\cite{carmona2009optimal} and the follow up paper~\cite{carmona2010market}. In the first paper, the authors study optimal stochastic control and carbon price formation (firms minimise abatement costs plus a terminal penalty under a given cap, and the equilibrium price is derived in closed form), while, in the second, they provide a comprehensive framework for the design of emissions trading markets, comparing different allocation rules.
More recently, Kollenberg and Taschini~\cite{kollenberg2019dynamic} studied the Market Stability Reserve (MSR) and analysed how the MSR affects price dynamics and banking behaviour.
All these models either take the regulatory framework as given, the cap, the allocation rule, or the MSR parameters are fixed, and they study the firms' equilibrium response. The regulator's optimisation problem is either absent or treated separately from the market equilibrium. \\
Our paper departs from this by making the regulator an active optimiser within the model. To this extent, the paper closest in spirit to ours on the regulatory side is the work of A\"id and Biagini~\cite{aid2023optimal}. They consider a regulator who dynamically allocates emission allowances to a set of firms. Firms face both idiosyncratic and common economic shocks on emissions, and they have linear-quadratic abatement costs.
The key modelling choice in this paper is a Stackelberg game structure: the regulator leads by choosing the allocation policy, and firms follow by optimising their abatement and trading strategies. Using variational methods, they first derive the market equilibrium in closed form as a function of the regulator's allocation, and then they solve the regulator's problem. 
Now, the main difference with our approach is the following. A\"id and Biagini work with a finite number of firms, and the regulator directly allocates allowances to each firm. In our model, we work with a continuum of firms (a mean-field limit) and the regulator does not allocate permits directly to firms. Instead, she issues permits into the market through an auction, and the price emerges endogenously from a market clearing condition. So the interaction between the regulator and the firms is mediated entirely by the permit price, which is a more realistic description of how the EU-ETS actually works.

Another important strand of literature that we build on is the work on mean-field games with endogenous prices determined by market clearing. The key references here are two papers by Fujii and Takahashi \cite{fujii2022equilibrium, fujii2022mean}. The first studies equilibrium pricing under a market clearing condition in a general MFG setting, showing that the equilibrium price can be characterised by a McKean–Vlasov FBSDE with common noise and proving convergence of the net order flow to zero in the large-$N$ limit. The second extends this framework to a setting with a major player who influences the price, yielding a Stackelberg-type interaction between a large player and a continuum of small players; this is conceptually related to our regulator, although the specific structure differs. On the theoretical side, we rely heavily on the foundational work of Carmona, Delarue, and Lacker \cite{carmona2016} which developed the existence and uniqueness theory for mean-field games with common noise using strong and weak solutions and proved an analogue of the Yamada–Watanabe theorem, precisely the tool we use to establish that the filtering error is independent of the regulator's control. 

The most direct predecessor of our work is the paper by Del Sarto, Leocata, and Livieri \cite{del2024mean}, which develops a mean-field game of competitive firms producing similar goods according to a standard AK model: exactly the production technology we adopt, with capital generating pollution as a by-product. They study a cap-and-trade regulation in which firms respond by implementing pollution abatement, reducing output, and trading permits, while a regulator dynamically allocates emission allowances. The key result of their paper is that the resulting MFG is of linear-quadratic type and equivalent to a mean-field type control problem, making it a potential game; they obtain explicit solutions through Riccati equations and characterise the carbon emission equilibrium price satisfying the market clearing condition via a McKean–Vlasov FBSDE with common noise. What their framework does not address, however, is the regulator's optimisation problem: in their model, the permit allocation is taken as given, and the regulator does not actively choose how many permits to issue. Our paper builds directly on their framework and closes this loop, taking the equilibrium characterised by Del Sarto et al.\ as the state dynamics for the regulator's problem and solving for the optimal permit policy within the mean-field control framework.

In this paper, we propose a novel framework that places the regulator’s decision-making process at the centre of the analysis. We consider a family of competitive firms operating under an AK production model, where output is directly proportional to the capital stock. Each firm’s production generates emissions and requires permits to operate, creating a direct dependence between the firms’ decisions and the number of permits issued by the regulator. The regulator, acting as a central planner, faces a fundamental trade-off: minimising average emissions across the market while maximising the aggregate output of firms. This problem is particularly challenging because the regulator’s control decisions, namely the quantity of permits issued at each time influence the firms’ strategies, which, in turn, shape market dynamics.

To address this, we frame the regulator’s decision-making as a stochastic optimal control problem. The state variable in the regulator’s problem is the solution to a system of forward-backward stochastic differential equations (FBSDEs), which describes the equilibrium behaviour of a typical firm in the market. Specifically, the forward component represents the dynamics of the firm’s production and emissions, while the backward component captures the optimal response of firms to permit allocations. By incorporating this feedback loop, we provide a unified framework that integrates the regulator’s objectives with the equilibrium strategies of competitive firms.

A key feature of our model is the real-time auction system through which the permits are allocated. This system dynamically adjusts the supply of permits to balance market efficiency with the regulator’s long-term objectives. The auction mechanism ensures that the firms’ decisions remain aligned with the regulator’s control strategy, providing a robust tool for managing the inherent uncertainties in market and environmental conditions.

The regulator’s optimal strategy is derived within the framework of Mean Field Control (MFC). By leveraging the linear-quadratic structure of the problem and incorporating common noise to model external uncertainty, we characterise the solution using a system of FBSDEs and Riccati equations. Our analysis provides both theoretical insights and practical guidelines for designing optimal regulatory policies in emission trading systems. 

This work offers a novel perspective on the interaction between decentralised decision-making by competitive firms and centralised regulation by an environmental authority. By coupling the firms’ equilibrium problem with the regulator’s stochastic optimal control problem, we shed light on the trade-offs inherent in emission trading systems and highlight the potential of real-time auction mechanisms for achieving both environmental and economic goals.

The paper is structured as follows. In Section~\ref{sect_model_N}, we consider the set-up for the N-player model of firms. In particular, given the potential structure of the game we find an explicit solution and we exploit this in Subsection~\ref{subsect_model_N:regulator_rewritten} to rewrite the regulator's problem introduced in Subsection \ref{subsect_model_N:regulator}. Section \ref{section: mean field approach} focuses on the mean-field limit approximation of the market, with particular attention, in Subsection \ref{subsection_mfg_regulator}, on the consequences in the structure of the market clearing condition in the limit.
In Section \ref{section_solution_MF_regulator_cpo}, we analyse the MF optimal control problem of the regulator whose state dynamics is given by a controlled FBSDEs system. In addition, we provide a characterisation for the solution in terms of  systems of Riccati ODEs. Finally, in Section \ref{section: numerical_experiments}
we exploit the explicit characterizations obtained in the previous sections to illustrate via carefully selected numerical illustration our findings. 
For the sake of completeness, the solution for the firms' problem in the unregulated-market case and a step-by-step self-consistent route to implementations are included in Appendix~\ref{app: unregulated market} and Appendix~\ref{app: sbs implementation} respectively.

\subsection*{Notation}
We denote a given filtered probability space with the standard notation $\setup$. We denote with $\mathcal{G}\otimes\mathcal{H}$ the product sigma-algebra generated by two sigma-algebras $\mathcal{G}$ and $\mathcal{H}$. 

For every probability space $(\Omega,\mathbb{G},\prob)$ endowed with a filtration $\mathbb{G}:= (\mathcal{G}_t)_{t\in[0,T]}$ we refer to:
\begin{itemize}
\item $\mathbb{L}^2(\G;\R)$ as the set of real valued $\G$-measurable square integrable random variables;
\item $\mathbb{S}^2(\mathbb{G};\R)$ as the set of real valued $\mathbb{G}$-adapted \cadlag  processes $X$ satisfying:
\begin{displaymath}
||X||_{\mathbb{S}^2(\mathbb{G};\R)}:= \E\bigg[ \sup_{t\in[0,T]} |X_t|^2 \bigg]^{\frac{1}{2}}<\infty;
\end{displaymath}
\item $\mathbb{H}^2(\mathbb{G};\R)$ is the set of real valued $\mathbb{G}$-progressively measurable processes $Z$ satisfying:
\begin{displaymath}
||Z||_{\mathbb{H}^2(\mathbb{G};\R)} := \E\bigg[ \bigg( \int_0^T |Z_t|^2 dt \bigg) \bigg]^{\frac{1}{2}}<\infty.
\end{displaymath}
\end{itemize}

\section{The model setup}\label{sect_model_N} 
\subsection{The companies problem}\label{subsect_model_N:companies} 
We consider a family of $N$ firms which populate a market over a finite time interval $[0,T]$. These firms follow an AK model, where output is proportional to capital, implying constant returns to capital, to maximise the production of a given product. We suppose that each firm $i \in \{1, \dots, N\}$ can choose to create new capital $K^i$ using a mix of fossil-fuel and green energy based level of capital. In particular, we assume that the capital level of firm $i$ at time $t$, denoted by $K^i_t$, follows the SDE:
\begin{equation}
\label{eqn: K dynamics}
\begin{cases}
dK^i_t = (a^f \alpha^{f,i}_t + a^g \alpha^{g,i}_t - \delta K^i_t) dt + \sigma^K K^i_t dW^{K,i}_t,\\
K^i_0 = K^0,
\end{cases}
\end{equation}
where $a^f,a^g,\delta, \sigma^K, K^0$ are positive constants. The quantities $\alpha^{f,i}$ and $\alpha^{g,i}$ represent the amount of fossil-fuel and green energy-based level of capital used by firm $i$ for capital creation. The parameter $\delta$ represents the depreciation rate of capital, while $\sigma^KK^i$ represents the standard deviation of the level of capital of firm $i$. The Brownian motions $(W^{K,i})_{i = 1,\dots,N}$ are supposed to be independent.
$\alpha^{f,i}$ and $\alpha^{g,i}$ represent the controls of firm $i$ on capital production. 

The process of capital production generates emissions, that are related to the fossil-fuel based level of capital applied by the companies. Denoting by $E^i_t$ the emission level of firm $i$ in the time interval $[0,t]$, we suppose that $E^i$ follows the dynamics
\begin{equation}
\label{eqn: E dynamics}
\begin{cases}
dE^i_t = (a^e \alpha^{f,i}_t - \alpha^{e,i}_t) dt + \sigma^E(\rho dW^{E,i}_t + \sqrt{1- \rho^2} dW^0_t),\\
E^i_0 = E^0,
\end{cases}
\end{equation}
where $a^e,\rho, \sigma^E, E^0$ are positive constants, $W^0$ is a Brownian motion independent of $W^{K,i}$ for every $i = 1,\dots,N$ and $(W^{E,j})_{j=1,\dots,N}$ is a family of pairwise independent Brownian motions, independent of $W^0$ and $W^{K,i}$ for every $i = 1,\dots,N$. $W^{E,i}$ is the idiosyncratic noise of the firm $i$ associated with emission level, while $W^0$ represents common shocks which affect all the companies' emission rate. The control $\alpha^{e,i}$ represents the effort rate implemented by company $i$ to reduce the emissions. As a consequence, we are assuming that $\{(W^{K,i},W^{E,i},W^{0})\}_{i=1,\dots,N}$ is a family of 3-dimensional Brownian motions.

To comply with the ETS-based carbon reduction policy, each company must maintain an allowance account in which emission permits are stored. The permits held in this account determine the amount of emissions that the company is allowed to produce. The dynamics of the allowance account of company $i$ are described by the stochastic process $X^i$, solution to the following SDE:

\begin{equation}
\label{eqn: bank account company i}
\begin{cases}
dX^i_t &= \alpha^{x,i}_t dt- dE^i_t,\\
&= (\alpha^{x,i}_t - a^e \alpha^{f,i}_t + \alpha^{e,i}_t) dt -\sigma^E(\rho dW^{E,i}_t + \sqrt{1- \rho^2} dW^0_t),\\
X^i_0 &= X^0 := A^0 - E^0.
\end{cases} 
\end{equation}
In the previous equation, $\alpha^{x,i}_t$ represents the number of permits traded by company $i$ in the infinitesimal interval $[t,t+dt]$, while $A^0$ represents the initial number of permits owned by the company. Hence, the emission level $E^i_t$ is measured in terms of the number of permits that the $i$-th company consumes in the time interval $[0,T]$ to create capital exploiting fossil-fuel-based energy. 

We introduce the cost function that has to be minimised by firms. In analogy to \cite{del2024mean}, we consider quadratic costs for capital production and abatement effort:
\begin{equation*}
C^p(\alpha^{p,i}_t):= c_{1,p} \alpha^{p,i}_t + c_{2,p} (\alpha^{p,i}_t)^2,\quad t\in[0,T],\quad p = f,g,e,
\end{equation*}
where $c_{i,p}>0$ for every $i = 1,2$ and $p = f,g,e$. 
We denote the empirical mean of a sequence of random variables $\underline{A}:= (A^1,\dots,A^N)$ by
\begin{equation}
\label{eqn: empirical mean notation}
\mathfrak{m}^{(N)}(\underline{A}) := \frac{1}{N} \sum_{j = 1}^N A^j.
\end{equation}
We suppose that firm $i$ faces a linear inverse demand function defined by
\begin{equation}
\label{eqn: price good economy n}
p(K^i_t,\mathfrak{m}^{(N)}(\underline{K}_t)):= a- b(1-\gamma) AK^i_t - b\gamma A\mathfrak{m}^{(N)}(\underline{K}_t),\quad t\in[0,T],
\end{equation}
where $\underline{K}_t:= (K^1_t,\dots, K^N_t)$, $a,b>0$ and $\gamma \in[0,1]$ describes the degree of production substitution, while $A$ represents the technological level so that $AK^i_t$ is the production function of the firm $i$.

Finally, we suppose that the trading activities of firm $i$, necessary to provide the best amount of permits needed at any time is linear quadratic in the number of permits $\alpha^{x,i}_t$ traded at every time $t\in[0,T]$. The trading cost is supposed to be: 
\begin{equation}
\label{eqn: quadratic cost functional}
C^x(\alpha^{x,i}_t,\w_t) := \w_t \alpha^{x,i}_t + c_{2,x} (\alpha^{x,i}_t)^2,\quad t\in[0,T],
\end{equation}
where $c_{2,x}>0$ represents the transaction costs for trading activities and $\w_t$ is the market price of the permits at time $t\in[0,T]$. We make now the following assumption
\begin{assumption}
\label{ass: price taker companies}
In analogy to \cite{del2024mean}, we assume that no company is large enough to influence the price process $\w$. In other words, all companies are assumed to be \emph{price takers}, and $\w$ is therefore regarded as an exogenous stochastic process from the perspective of the firms, possibly correlated with all sources of randomness in the market.
\end{assumption}
To be more rigorous, we introduce the model setup. It is given by a filtered probability probability space $(\Omega,\ItF,\prob,\F)$, on which the following processes are defined:
\begin{itemize}
\item A one-dimensional Brownian motion $W^0$, playing the role of the common noise;
\item A family of $N$ idiosyncratic noises describing the randomness in the production dynamics $\underline{W}^{N;K} := (W^{K,1},\dots ,W^{K,N})$.
\item A family of $N$ idiosyncratic noises describing the randomness in the emission dynamics $\underline{W}^{N;E} := (W^{E,1},\dots ,W^{E,N})$.
\end{itemize}
We suppose that $(W^0,\underline{W}^{N;K},\underline{W}^{N;E})$ is a $(2N+1)$-dimensional Brownian motion.

For the moment, the price $\w$ is a stochastic process defined on $(\Omega,\ItF,\prob,\F)$. It is natural to suppose that each company $i$ has access to the information given by $W^0$, $\w$, and $W^i:=(W^{K,i},W^{E,i})^\top$. As a consequence, adopting the notation
\begin{align*}
v^i_t 			&:= (\alpha^{f,i}_t,\alpha^{g,i}_t,\alpha^{e,i}_t,\alpha^{x,i}_t),\quad t\in[0,T],\\
\mathcal{X}^i_t &:= (K^i_t, X^i_t), \quad t\in[0,T],
\end{align*}
for respectively the control and the state variable of firm $i = 1,\dots,N$, every firm has to minimise the following target function
\begin{align}
\label{eqn: target companies N}
&J(v^i;\mathcal{X}^i,\mathfrak{m}^{(N)}(\underline{K}), \w) \\
\nonumber& \quad:= \E\Bigg[\int_0^T \bigg[-p(K^i_t,\mathfrak{m}^{(N)}(\underline{K}_t)) AK^i_t + \alpha^{x,i}_t \w_t + c_{2,x}(\alpha^{x,i}_t)^2 + \sum_{p = f,g,e}C^p(\alpha^{p,i}_t)\bigg] dt + \lambda (X^i_T)^2\Bigg],
\end{align}
where $-p(K^i_t,\mathfrak{m}^{(N)}(\underline{K}_t)) AK^i_t$ represents the revenues coming from capital production, and $\lambda(X^i_T)^2$ is a final penalisation term present to avoid accumulation of permits in the bank accounts of the firms. In addition, we suppose that the class of admissible controls for the company $i$-th is
\begin{equation*} 
v^i \in \mathbb{A}^i := \mathbb{H}^2( \F^{W^0, W^{K,i}, W^{E,i} ,\w};\R^4),\quad i = 1,\dots, N.
\end{equation*}

Finally, the state variable of firm $i$ is fully determined by the following two dimensional SDE: 
\begin{equation*}
\begin{cases}
dK^i_t &= (a^f \alpha^{f,i}_t + a^g \alpha^{g,i}_t - \delta K^i_t) dt + \sigma^K K^i_t dW^{K,i}_t, \quad K^i_0=K^0.\\
dX^i_t &= (\alpha^{x,i}_t - a^e \alpha^{f,i}_t + \alpha^{e,i}_t) dt -\sigma^E(\rho dW^{E,i}_t + \sqrt{1- \rho^2} dW^0_t), \quad X^i_0=X^0.\\
\end{cases}
\end{equation*}
Adopting the notation of \eqref{eqn: empirical mean notation}, for every $\mathcal{X}^i\in \R^2,\  v^i\in \R^4$, for all $i = 1,\dots,N$, for all $\w\in\R$ we define:
\begin{align*}
f(\mathcal{X}^i,v^i;\mathfrak{m}^{(N)}(\mathcal{X}),\w) &:= -p(K,\mathfrak{m}^{(N)}(\underline{K})) AK^i +\sum_{p \in\{ f,g,e,x\}} C^p(\alpha^{p,i}) \\
&\qquad-\Big[ a-b(1-\gamma)AK^i - b\gamma A\mathfrak{m}^{(N)}(\underline{K})\Big]AK^i\\
&\qquad+  \sum_{p \in \{f,g,e\} } \Big[c_{1,p}\alpha^{p,i} + c_{2,p} (\alpha^{p,i})^2\Big]+ \w\alpha^{x,i} + c_{2,x} (\alpha^{x,i})^2\\
 &= \langle Q_1 \mathcal{X}^i ,\mathcal{X}^i\rangle+ 2\langle {Q}_2 \mathfrak{m}^{(N)}(\underline{\mathcal{X}}), \mathcal{X}^i\rangle + \langle Q_3 v^i, v^i\rangle + 2\langle q_1, \mathcal{X}\rangle + 2 \langle q_2(\w),v^i\rangle.\\
g(\mathcal{X}^i)&:= \langle Q_4\mathcal{X}^i, \mathcal{X}^i\rangle.
\end{align*}
where $Q_1,Q_2,Q_4\in M_{2,2}(\mathbb{R})$ are positive semi-definite, $Q_3\in M_{4,4}(\mathbb{R})$ is positive definite, $q_1\in\mathbb{R}^2$ and $q_2(\w)\in\mathbb{R}^4$, the latest depending on $\w$:
\begin{align*}
Q_1&:= 
\begin{pmatrix} 
b(1-\gamma) A^2 & 0 \\
0&0
\end{pmatrix},
\quad {Q}_2 :=
\begin{pmatrix} 
\frac{1}{2}b\gamma A^2 &0\\
0&0
\end{pmatrix},
\  Q_3:= 
\begin{pmatrix} 
c_{2,f} 	&0			&0			&0\\
0		&c_{2,g} 		&0			&0\\
0		&0			&c_{2,e}		&0\\
0		&0			&0			&c_{2,x}
\end{pmatrix},
\\
Q_4&:=
\begin{pmatrix}
 0&0\\
 0&\lambda
 \end{pmatrix} ,\ 
q_1 := 
\begin{pmatrix} 
-\frac{1}{2}aA \\
0
\end{pmatrix}
\  q_2(\w) := 
\begin{pmatrix}
\frac{1}{2}c_{1,f}\\
\frac{1}{2}c_{1,g}\\
\frac{1}{2}c_{1,e}\\
\frac{1}{2}\w
\end{pmatrix}.
\end{align*}
On the other hand, the state variable of the firm $i$ can be rewritten as in \cite{del2024mean}:
\begin{equation}
\label{eqn: state variable N}
\begin{cases}
d\mathcal{X}^i_t = \big( A_1 \mathcal{X}^i_t + A_2 v^i_t\big) dt + \sum_{j \in\{ K,E\}} (A^j_3\mathcal{X}^i_t+ A^j_4)dW^{j,i}_t + A_5dW^0_t,\\
\mathcal{X}^i_0 = (K_0,X_0),
\end{cases}\quad t\in[0,T],
\end{equation}
where for the company $i$, $W^i := (W^{K,i}, W^{E,i})$ and
\begin{align*}
A_1 &:= 
\begin{pmatrix}
-\delta	&0	\\
0		&0	
\end{pmatrix},
\ 
A_2 = 
\begin{pmatrix} 
a^f	&a^g		&0	&0\\
-a^e	&0		&1	&1
\end{pmatrix},
\ A^K_3=
\begin{pmatrix} 
\sigma^K	&0	\\
0		&0
\end{pmatrix},\\
\ A^K_4 &= \mathds{O},
\ A^E_3 = \mathds{O}, \  A^E_4 = \begin{pmatrix}
0\\
-\sigma^E\rho
\end{pmatrix},\
A_5 = \begin{pmatrix}
0\\
-\sigma^E\sqrt{1-\rho^2}
\end{pmatrix}.
\end{align*}
In analogy to \cite{del2024mean}, we denote the volatility function by
\begin{align*}
    \sigma(\mathcal{X}) &:= 
    \begin{pmatrix}
        A^K_3\mathcal{X} & A^E_4  
    \end{pmatrix}
    =
    \begin{pmatrix} 
    \sigma^K K	&0	\\
    0		& -\sigma^E \rho
    \end{pmatrix}\in M_{2\times 2}(\R) ,\\  
    \sigma^0 &:= A_5 \in M_{2\times 1}(\R).
\end{align*}
By Lemma \cite[Lemma 1.56]{carmona2018probabilistic2} and \cite[equation (6.82)]{carmona2018probabilistic2}, for all processes $\bar{\xi}, \w \in \mathbb{H}^2([0,T],\R)$, there exists a unique control $\widehat{v}^i$ such that $J(\widehat{v}^i;\mathcal{X}^i,\bar{\xi},\w)= \inf_{v \in \mathbb{H}^2(\mathbb{F}^i;\mathbb{R}^4)}J(v;\mathcal{X}^i,\bar{\xi},\w)$. 
Indeed, for every
\begin{equation*}
    \mathcal{Y}: = \begin{pmatrix}V\\ Y\end{pmatrix}\in\R^2,\quad Z:= \begin{pmatrix} Z^{1,1}&Z^{1,2}\\Z^{2,1}&Z^{2,2}\end{pmatrix}\equiv \begin{pmatrix}
        Z^K, Z^E
    \end{pmatrix}\in M_{2\times2}(\R) ,\quad Z^0 := \begin{pmatrix}Z^{0,1}\\ Z^{0,2}\end{pmatrix} \in \R^2,
\end{equation*}
the full Hamiltonian satisfies the following structure: 
\begin{align*}
&H(v,\mathcal{X}, \mathcal{Y},Z^0,Z;\bar{\xi},\w)= \\
&\qquad =  \langle A_1 \mathcal{X} + A_2v,\mathcal{Y}\rangle + f(\mathcal{X}, v; \bar{\xi}, \w)	+ \text{trace}\bigg\{\bigg( \sigma(\mathcal{X}),A_5\bigg)\begin{pmatrix}
    Z^\top\\
    (Z^0)^\top\end{pmatrix}\bigg\}\\
&\qquad= \langle A_1 \mathcal{X} + A_2v,\mathcal{Y}\rangle + f(\mathcal{X}, v; \bar{\xi}, \w)	+ \text{trace}\bigg( \sigma( \mathcal{X} ) Z^\top +  A_5(Z^0)^\top\bigg) 	\\
&\qquad= (a^f\alpha^{f} + a^g\alpha^{g}- \delta K)V + (\alpha^x -a^e\alpha^f + \alpha^e) Y + \langle Q_1 \mathcal{X} ,\mathcal{X}\rangle+ 2\langle Q_2\overline{\xi}, \mathcal{X}\rangle + \langle Q_3 v, v\rangle  \\
&\qquad\qquad+2\langle q_1, \mathcal{X}\rangle + 2 \langle q_2(\w),v\rangle + \text{trace}\bigg\{ 
\begin{pmatrix}
    \sigma^K K Z^{1,1} & \sigma^K K Z^{2,1}\\
    -\sigma^E \rho Z^{1,2} & -\sigma^E \rho Z^{2,2} 
\end{pmatrix} \\
&\qquad\qquad+ 
\begin{pmatrix}
    0&0\\
    -\sigma^E\sqrt{1-\rho^2}Z^{0,1}&-\sigma^E\sqrt{1-\rho^2}Z^{0,2}
\end{pmatrix}\bigg\}\\ 
&\qquad = (a^f\alpha^{f} + a^g\alpha^{g}- \delta K)V + (\alpha^x -a^e\alpha^f + \alpha^e) Y + \langle Q_1 \mathcal{X} ,\mathcal{X}\rangle+ 2\langle Q_2\overline{\xi}, \mathcal{X}\rangle + \langle Q_3 v, v\rangle  \\
&\qquad\qquad+2\langle q_1, \mathcal{X}\rangle + 2 \langle q_2(\w),v\rangle + \sigma^KKZ^{1,1} - \sigma^E [ \rho Z^{2,2} + \sqrt{1-\rho^2} Z^{0,2}].
\end{align*}
As a consequence, the first order condition becomes
\begin{equation}
\label{eqn: optimal control}
2Q_3 \widehat{v} + 2q_2(\w)+ A_2^\top \mathcal{Y}= 0,\Longrightarrow \widehat{v}:= -Q_3^{-1} \bigg(q_2(\w)+\frac{1}{2}A_2^\top \mathcal{Y}\bigg), 
\end{equation}
which is explicitly given by:
\begin{equation}
\label{eqn: FOC} 
\begin{cases}
\widehat{\alpha}^f &:= -\frac{1}{2c_{2,f}}(a^fV -a^eY + c_{1,f}),\\
\widehat{\alpha}^g &:=-\frac{1}{2c_{2,g}}(a^gV + c_{1,g}),\\
\widehat{\alpha}^e&:= -\frac{1}{2c_{2,e}}(c_{1,e} + Y),\\
\widehat{\alpha}^x&:= -\frac{1}{2c_{2,x}}(\w + Y).
\end{cases}
\end{equation}

In particular, the optimal control is $\widehat{v}_t=-Q_3^{-1}(q_2(\w_t)+\frac{1}{2}A_2^\top \mathcal{Y}_t)$, where $\mathcal{Y}_t$ solves the backward stochastic differential equation:

\begin{equation} 
\label{eqn: adjoint process N}
\begin{cases}
d\mathcal{Y}^i_t = -\big( A_1^\top \mathcal{Y}^i_t + (A^K_3)^\top Z^{i,K}_t + 2Q_1\mathcal{X}^i_t + 2Q_2 \mathfrak{m}^{(N)}(\underline{\mathcal{X}}_t) + 2q_1 \big)dt + Z^i_t dW^i_t + Z^{0,i}_t dW^0_t ,\\
\mathcal{Y}^i_T = Q_4\mathcal{X}^i_T.
\end{cases}
\end{equation} 

\subsection{The regulator problem}
\label{subsect_model_N:regulator}
As discussed in the introduction, the model setup considered in this paper is analogous to the one analysed in \cite[Section 3]{del2024mean}. In contrast to the cited reference, our focus is on the optimal regulatory strategy aimed at balancing carbon-emissions reduction and economic sustainability. 

The ETS mechanism is based on an auction system in which the regulator issues emission permits. The market price $\varpi$ is then determined by the interaction between firms’ demand for permits and the number of permits issued by the regulator. Consequently, by deciding over time how many \emph{new} permits to release into the market, the regulator can influence the price process and, in turn, the firms’ optimal decisions in the maximisation of their objective functionals.

\subsubsection{The auction system}
We describe how the auction system works. We consider the equilibrium price, determined by the balance between the demand (number of permits requested by the companies) and the supply (number of permits sold by the companies as well as issued by the regulator in the market). This condition, called \emph{market clearing} is defined by
\begin{equation}
\label{eqn: market clearing N market}
\sum_{j = 1}^{N} \alpha^{x,j}_t + \beta^{(N)}_t = 0,\quad dt\otimes \mathbb{P},
\end{equation} 
where the control $\beta^{(N)}:= (\beta^{(N)}_t)_{t\in[0,T]}$  represents the (total) number of permits issued at time $t$ by the regulator.

\begin{remark}
\label{rmk:beta-sign}
The number of permits issued by the regulator should be of the same order of magnitude of the number of companies in the market. This is convenient to guarantee that the effect of the regulator's policies does not become negligible when the number of companies is large (for details see \cite[Remark 3.11]{fujii2022equilibrium}). As a consequence, we assume that the number of permits, in the market populated by $N$ companies, is of the form $\beta^{(N)}_t \equiv N {\beta}_t$, where $\beta_t$ denotes the \emph{normalised number of permits} issued by the regulator. 
In addition, from an economic point of view it is natural to suppose that the regulator cannot withdraw permits from the market, but it can only issue them. As a consequence, $\beta_t\leq0$ for every $t\in[0,T]$, i.e. $A= [-a,0]$ (eventually we can assume that $A = (-\infty, 0]$). 
\end{remark}

The companies are assumed to be price takers. Therefore, when solving their individual optimal control problem, they treat the price process $\w$ as exogenous. Hence, the stochastic maximum principle derived above applies and the optimal control $\hat{\alpha}^{x,j}_t$ of every company, $j \in \{1, \dots, N\}$, can be described by the last component of the first order condition on the Hamiltonian: $\widehat{\alpha}^{x,j}_t= -\frac{1}{2c_{2,x}}(\w_t + Y^j_t)$. Substituting this control in \eqref{eqn: market clearing N market}, the equilibrium price solves
\begin{equation*} 
    \sum_{j = 1}^{N} -\frac{1}{2c_{2,x}}(\w_t + Y^j_t) + \beta^{(N)}_t = 0.
\end{equation*}
As discussed in \cite{follmer1974random}, when the number of companies is sufficiently large, it is convenient to consider the normalised version of \eqref{eqn: market clearing N market}, obtained by dividing \eqref{eqn: market clearing N market} by $N$. We then have 
\begin{equation} 
\label{eqn: eq price process N}
    \w^N_t 
    \equiv -\frac{1}{N}\sum_{j = 1}^{N}Y^j_t + 2c_{2,x}\frac{1}{N}\beta^{(N)}_t 
    = -\mathfrak{m}^{(N)}(\underline{Y}_t) + 2c_{2,x}\beta_t. 
\end{equation}
This definition is in accordance with Remark \ref{rmk: controls reg N} below. Indeed, the price $\w^N$ depends on $\beta^{(N)}$ but also on 
$\mathfrak{m}^{(N)}(\underline{Y}_t)$. 
This suggests that assuming that $\beta^{(N)}$ only depends on the common shock $W^0$ would be overly restrictive. 

We adopt the notation 
\begin{equation*}
\mathfrak{m}^{(N)}(\underline{\mathcal{Y}}_t) := \mathfrak{m}^{(N)}((\underline{V}_t,\underline{Y}_t)^\top) =\frac{1}{N} \sum_{j = 1}^N (V^j_t,Y^j_t)^\top.
\end{equation*}
If the price satisfies \eqref{eqn: eq price process N}, we have an interdependent structure given by the empirical mean of the optimal controls of the companies as well as the control of the regulator. Considering the solution $\mathcal{X}$ of \eqref{eqn: state variable N}, where the control $v$ is the candidate optimal control introduced in \eqref{eqn: FOC}, we observe that the drift of $\mathcal{X}$ depends on $\w^N$. Hence, substituting $\w^N$ in the form \eqref{eqn: eq price process N} in \eqref{eqn: FOC}, we conclude that $\mathcal{X}$ depends on $\beta^{(N)}$.


\subsubsection{The regulator's target}

In this setting, we introduce an optimal control problem for the regulator, who plays the role of a central planner. We recall that $K^j_t$ represents the capital level produced at time $t$ by company $j = 1,\dots,N$, while $E^j_t$ represents the cumulated level of emission of company $j$. Hence, the regulator is not solely interested in minimising total emissions, but rather in optimising the trade-off between environmental sustainability (i.e., reducing emissions) and industrial sustainability (i.e., supporting the economic performance of firms). 
From an economic point of view, a realistic objective regarding industrial sustainability is the maximisation of average profit of the firms, which depends on their production, energy composition, and emission abatement strategies. This approach follows the same reasoning as in \cite[equation (8)]{aid2023optimal}, where the policy maker seeks to optimise social welfare derived from firm-level performance. Accordingly, in the finite-dimensional market we define the regulator’s objective functional as
\begin{align}
\label{eqn: target-reg-N}
    J^0_{(N)}(\beta^{(N)})
    & :=\E\Bigg[\frac{1}{N}\sum_{i = 1}^N \Bigg\{ \int_0^T a(t)\bigg[
    b(t)\beta_t^2-p(K^i_t, \mathfrak{m}^{(N)}(\underline{K}_t)) A K^i_t 
    + \hat{\alpha}^{x,i}_t \, \w^N_t + c_{2,x}(\hat{\alpha}^{x,i}_t)^2 
     \\
    \nonumber&\qquad
    + \sum_{p = f,g,e} C^p(\hat{\alpha}^{p,i}_t)
    \bigg] dt+ a(T)\lambda (X^i_T)^2 + \Bigg( \avg E^j_T - \theta\Bigg)^2 \Bigg\} \Bigg],
\end{align}

where $\w^N$ is defined in \eqref{eqn: eq price process N} and $\theta$ denotes the regulator’s target level for average economy-wide emissions after the implementation of carbon-reduction policies.

In its most straightforward formulation, the regulator’s objective combines two effects: the maximisation of firm profits and the minimisation of emissions. However, these two effects enter the functional with different growth structures. Indeed, firms’ profits depend quadratically on the production level through the revenue term, whereas cumulative emissions enter linearly. As a consequence, the incentive to sustain production tends to dominate the incentive to reduce emissions. This imbalance reflects the absence of an explicit representation of the \emph{cap} side of the cap-and-trade mechanism, i.e., the gradual limitation of the total number of permits over time. Consequently, the regulator faces a strong bias towards issuing more permits to support production, with only a relatively weaker incentive to curb emissions.
To prevent this unrealistic behaviour, we introduced a positive quadratic term $b(t)\beta_t^2$, with  $b(t) > c_{2,x}$ for all $t\in[0,T]$, which is included in the objective functional to penalize excessive new permits issuance. This term represents an implicit cost borne by the regulator when a high number of new permits is issued. Choosing $b(t)$ as an increasing function over time reinforces the interpretation that environmental policies become progressively more stringent as environmental concerns grow. In this formulation, the regulator evaluates the aggregate economic performance of the firms through their expected profits, while internalising the environmental cost generated by their emissions. The resulting policy $\beta$ determines the number of new emission permits to be issued effectively steering the trade-off between the stimulation of industrial growth and the limitation of environmental degradation.

In addition, in order to quantify the effect of the ETS-based policy of the regulator, the value $\theta$ should be defined starting from the expected amount of CO2 emitted by the economy in absence of policy, e.g. as $50\%$ of the CO2 emitted under the same economic conditions, but with no policy applied by the regulator.

\begin{remark}
\label{rmk: controls reg N}
In a market populated by finitely many companies, the strategy of the regulator should depend on any source of randomness: on the common shocks as well as the idiosyncratic shocks of every agent. However, when the number of companies is large, it is unrealistic to expect that the regulator takes her strategies in response to the noises of every company. It is actually more realistic to suppose that she can observe the common shocks affecting the economy as well as the equilibrium price process, which implicitly depends on every idiosyncratic term. Hence, the proper class of controls for the regulator should be $\mathbb{A}^{0,N} : =\mathbb{H}^2( \F^{W^0,\w^N};\R)$, where $\F^{W^0,\w^N}$ is the filtration generated by $W^0$ and $\w^N$, which is possibly correlated with the idiosyncratic noises $\underline{W}^{N;E}, \underline{W}^{N;K}$, through $\w^N$. As discussed in \cite[Remark 3.1]{fujii2022equilibrium}, this is a key issue, which leads the solution of the optimal control problem for the regulator in the finite dimensional case very difficult to solve. As we are going to discuss in the next section, a natural strategy to circumvent this problem is to pass to the mean-field formulation of such model. 
\end{remark}

Finally, we introduce the following objects to emphasise the dependence of the FBSDE on the control of the regulator:
\begin{equation*}
q_3 := 
\begin{pmatrix}
\frac{1}{2}c_{1,f}\\
\frac{1}{2}c_{1,g}\\
\frac{1}{2}c_{1,e}\\
0
\end{pmatrix},\quad 
Q_5=
\begin{pmatrix}
0	\\
0	\\
0	\\
c_{2,x}
\end{pmatrix},
\quad Q_6 :=
\begin{pmatrix}
0	&0\\
0	&0\\
0	&0\\
0	&-\frac{1}{2}\\
\end{pmatrix},
\end{equation*}
As a consequence, FBSDE \eqref{eqn: state variable N}, \eqref{eqn: adjoint process N} becomes
\begin{equation}
\label{eqn: eq FBSDE company i}
\begin{cases}
d\mathcal{X}^i_t = \big[ A_1 \mathcal{X}^i_t + A_2( -Q_3^{-1} ( q_3 + Q_5{\beta}_t + Q_6\mathfrak{m}^{(N)}(\underline{\mathcal{Y}}_t)+ \frac{1}{2}A_2^\top \mathcal{Y}^i_t))\big] dt + \sigma(\mathcal{X}^i_t) dW^i_t + A_5dW^0_t,\\
d\mathcal{Y}^i_t = -\big( A_1^\top \mathcal{Y}^i_t + (A^K_3)^\top Z^{i,K}_t + 2Q_1\mathcal{X}^i_t + 2Q_2 \mathfrak{m}^{(N)}(\underline{\mathcal{X}}_t) + 2q_1 \big)dt + Z^i_t dW^i_t + Z^{0,i}_t dW^0_t ,\\
\mathcal{X}^i_0 = (K_0,X_0)^\top,\\
\mathcal{Y}^i_T = Q_4\mathcal{X}^i_T.
\end{cases}
\end{equation}

In this formulation, the regulator’s problem is convex in the regulator's control, ensuring the existence of an optimal solution while preserving a consistent economic interpretation. 
From a mathematical standpoint, this indeed guarantees the direct application of the stochastic maximum principle as presented, for example, in \cite[Appendix A]{fujii2022equilibrium}, since the problem does satisfy the standard convexity assumptions required for optimality.

\subsection{Rewriting the regulator's objective}\label{subsect_model_N:regulator_rewritten}

Let us now analyse every term of \eqref{eqn: target-reg-N}, to rewrite more explicitly the target function of the regulator. 
In order to simplify the notation we introduce the quantity $\beta_t:=\frac{1}{N}\beta^{(N)}_t$, for all $t \in [0,T]$.
\begin{align*}
    \frac{1}{N}\sum_{j = 1}^N p(K^j_t, \mathfrak{m}^{(N)}(\underline{K}_t)) AK^j_t & = \frac{1}{N}\sum_{j = 1}^N\bigg(a- b(1-\gamma) AK^j_t - b\gamma A\mathfrak{m}^{(N)}(\underline{K}_t)\bigg) AK^j_t\\
    & = aA\mathfrak{m}^{(N)}(\underline{K}_t) - b(1-\gamma) A^2\frac{1}{N} \sum_{j = 1}^N (K^j_t)^2 - b\gamma  A^2(\mathfrak{m}^{(N)}(\underline{K}_t))^2,\\
    \frac{1}{N}\sum_{j = 1}^NC^f(\hat\alpha^{f,j}_t) 
    & =\frac{1}{4 c_{2, f}}\bigg(-c_{1,f}^2 + \avg(a^f{V}^j_t-a^e{Y}^j_t)^2\bigg),\\
    \frac{1}{N}\sum_{j = 1}^NC^e(\hat\alpha^{e,j}_t)
    &= \frac{1}{4 c_{2, e}}(-c_{1,e}^2 + \frac{1}{N}\sum_{j=1}^N(Y^j_t)^2),\\
    \avg C^g(\hat\alpha^{g,j}_t) 
    &= \frac{1}{4 c_{2, g}}\bigg(-c_{1,g}^2 + (a^g)^2\frac{1}{N}\sum_{j = 1}^N (V^j_t)^2\bigg),\\
    \avg \big( \hat \alpha^{x,j}_t \w^N_t + c_{2,x} (\hat\alpha^{x,j}_t)^2)&= \avg \Bigg( - \frac{1}{2c_{2,x}} \w^N_t(\w^N_t + Y^j_t) + \frac{1}{4c_{2,x}} (\w^N_t + Y^j_t)^2\Bigg)\\
    &= \frac{1}{4c_{2,x}}\Bigg(\avg(Y^j_t)^2  - (\w^N_t)^2\Bigg) \\
    & = \frac{1}{4c_{2,x}}\Bigg[ \avg(Y^j_t)^2 - (\mathfrak{m}^{(N)}(\underline{Y}_t))^2 + 4c_{2,x} \mathfrak{m}^{(N)}(\underline{Y}_t)\beta_t - 4(c_{2,x})^2 \beta^2_t\Bigg],\\
    \bigg(\avg E^j_T - \theta\bigg)^2 
    & = \bigg(\int_0^T \avg \hat{\alpha}^{j,i}_t dt - \avg X^j_T + A^0-\theta\bigg)^2\\
    & = \bigg (\int_0^T \beta_tdt\bigg)^2 + ( \mathfrak{m}^{(N)}(\underline{X}_T))^2 + (A^0 - \theta)^2 \\
    &\qquad + 2\bigg(\int_0^T \beta_tdt\bigg)\mathfrak{m}^{(N)}(\underline{X}_T) - 2(A^0 - \theta) \mathfrak{m}^{(N)}(\underline{X}_T)\\
    &\qquad- 2(A^0-\theta) \bigg(\int_0^T\beta_t dt\bigg).
\end{align*}

In conclusion, adopting the notation $(A^j)_{j= 1}^N$, $\mathfrak{m}^{(2,N)}(\underline{A}_t) := \frac{1}{N} \sum_{j =1}^N (A^j_t)^2$ for $t\in[0,T]$, the target function for the central planner is 
\begin{align*}
    &J^0_{(N)}(\beta^{(N)}) \\
    &= \E\Bigg[\frac{1}{N}\sum_{i = 1}^N \bigg[ \int_0^Ta(t)\bigg[b(t)\left(\frac{1}{N}\beta^{(N)}_t\right)^2 -p(K^i_t, \mathfrak{m}^{(N)}(\underline{K}_t)) AK^i_t + \hat\alpha^{x,i}_t \w^N_t + c_{2,x}(\hat\alpha^{x,i}_t)^2 \\
    &\qquad + \sum_{p = f,g,e}C^p(\hat\alpha^{p,i}_t)\bigg] dt + a(T)\lambda (X^i_T)^2 + \bigg( \avg E^j_t - \theta\bigg)^2\bigg] \Bigg]\\
    &= \E\bigg[ \int_0^Ta(t) \bigg\{-aA\mathfrak{m}^{(N)} (\underline{K}_t) + b(1-\gamma) A^2 \mathfrak{m}^{(2,N)}(\underline{K}_t) + b\gamma A^2 (\mathfrak{m}^{(N)}(\underline{K}_t))^2 \\
    &\qquad + \mathfrak{m}^{(2,N)}(\underline{Y}_t) \bigg[\frac{(a^e)^2}{4c_{2,f}} + \frac{1}{4c_{2,e}} + \frac{1}{4c_{2,x}}\bigg] +  \mathfrak{m}^{(2,N)}(\underline{V}_t) \bigg[\frac{(a^f)^2}{4c_{2,f}} + \frac{(a^g)^2}{4c_{2,g}}\bigg]\\
    &\qquad - \frac{a^fa^e}{2 c_{2,f}} \avg (V^j_tY^j_t) - \frac{1}{4c_{2,x}}(\mathfrak{m}^{(N)}(\underline{Y}_t))^2 + \mathfrak{m}^{(N)}(\underline{Y}_t)\beta_t +(b(t)- c_{2,x})\beta^2_t \bigg\}dt \\
    &\qquad+ a(T) \lambda \mathfrak{m}^{(2,N)}(\underline{X}_T) +\bigg(\int_0^T\beta_tdt\bigg)^2 + ( \mathfrak{m}^{(N)}(\underline{X}_T))^2 + (A^0 - \theta)^2 \\
    &\qquad+ 2\bigg(\int_0^T \beta_tdt\bigg)\mathfrak{m}^{(N)}(\underline{X}_T) - 2(A^0 - \theta) \mathfrak{m}^{(N)}(\underline{X}_T) - 2(A^0-\theta) \bigg(\int_0^T\beta_t dt\bigg)\bigg].
\end{align*}
We introduce the cumulated-issued credit process: 
\begin{equation}
    \label{eqn: process Bt}
    B_t := \int_0^t \beta_sds.
\end{equation}
This additional controlled process should be included in the definition of the state variable of the regulator.

\section{A mean-field limit approximation of the market}
\label{section: mean field approach}
In this setting it seems natural to pass to the mean-field limit of the game, considering the limit in the number of companies populating the market. Indeed, the equations, that we derived for all the players of the market model, depend on the state and control of the considered player, together with the empirical mean of the solutions of \eqref{eqn: eq FBSDE company i}. 

We consider a filtered probability space $(\Omega, \ItF,\prob,\F)$ where $\F:= (\ItF_t)_{t\in[0,T]}$ is the usual augmentation of the filtration generated by a $3$-dimensional Brownian motion $(W^0,W)$, where $W := (W^E,W^K)$.

\subsection{The mean-field limit of the Nash equilibrium for the companies}
\label{subsection: mf companies game}
As in \cite{lacker2019mean}, we introduce the following random variable 
\begin{equation*}
\zeta:= (a^f,a^g,a^e,\delta,\sigma^K,\sigma^E,\rho,c_{1,f},c_{2,f},c_{1,g},c_{2,g},c_{1,e},c_{2,e},c_{1,f},c_{2,f},a,b,\gamma,A,K_0, X_0,\w,\lambda),
\end{equation*} 
that is called type vector and takes values in 
\begin{equation*}
\mathcal{Z}:= \R\times (0,\infty)^6 \times(0,1) \times (0,\infty)^{12} \times \R^2 \times (0,\infty)^2.
\end{equation*}
We call type distribution $m = \mathcal{L}(\zeta)$. 
Formally, as discussed in \cite[Section 3]{lacker2019mean}, the mean-field game is defined using the vector $\zeta$ to define a \emph{representative agent's} state variable, that solves: 
\begin{equation}
\label{eqn: mf state variable}
\begin{cases}
d\mathcal{X}_t = ( A_1 \mathcal{X}_t + A_2v_t)dt + \sigma(\mathcal{X}_t)dW_t + A_5dW^0_t,\\
\mathcal{X}_0=( K_0,X_0).
\end{cases}
\end{equation}
The target function to be minimised is
\begin{equation}
\label{eqn: mf target function}
J(v; \mathcal{X}, \bar{\mathcal{X}}, \w) := \E\bigg[\int_0^T \bigg(\langle Q_1 \mathcal{X}_t ,\mathcal{X}_t\rangle+ 2\langle Q_2 \bar{\mathcal{X}}_t, \mathcal{X}_t\rangle + \langle Q_3 v_t, v_t\rangle + 2\langle q_1, \mathcal{X}_t\rangle + 2 \langle q_2,v_t\rangle \bigg)dt + \langle Q_4\mathcal{X}_T, \mathcal{X}_T\rangle\bigg],
\end{equation}
where $\bar{\mathcal{X}}_t := (\E[K_t \mid \mathcal{F}^0_t],\E[X_t\mid \mathcal{F}^0_t])^\top$, where $(\mathcal{F}^0_t)_{t\in[0,T]}$ is the usual augmentation of the natural filtration of $W^0$. As a consequence, in analogy to \cite[Section 4]{del2024mean}, the optimal control, minimising the objective function $J(v; \mathcal{X}, \bar{K}, \w)$, is given by \eqref{eqn: optimal control}, where $\mathcal{Y}$ solves the following BSDE: 
\begin{equation} 
\label{eqn: mf adjoint process}
\begin{cases}
d\mathcal{Y}_t = -\big( A_1^\top \mathcal{Y}_t + (A^K_3)^\top Z^{K}_t + 2Q_1\mathcal{X}_t + 2Q_2 \bar{\mathcal{X}}_t + 2q_1 \big)dt + Z_t dW_t + Z^{0}_t dW^0_t ,\\
\mathcal{Y}_T = Q_4\mathcal{X}_T.
\end{cases}
\end{equation} 
In analogy to \cite[Proposition 4.2]{del2024mean}, strong existence and uniqueness of solutions of the FBSDE system
\begin{equation}
\label{eqn: mf FBSDE company}
\begin{cases}
d\mathcal{X}_t = \Big[A_1\mathcal{X}_t + A_2(-Q_3^{-1} (q_2(\w_t)+\frac{1}{2}A_2^\top \mathcal{Y}_t))\big]dt + \sigma(\mathcal{X}_t)dW_t + A_5dW^0_t,\\
\mathcal{X}_0=( K_0,X_0),\\
d\mathcal{Y}_t = -\Big[ A_1^\top \mathcal{Y}_t + (A^K_3)^\top Z^{K}_t + 2Q_1\mathcal{X}_t + 2Q_2 \bar{\mathcal{X}}_t + 2q_1 \Big]dt + Z_t dWt + Z^{0}_t dW^0_t ,\\
\mathcal{Y}_T = Q_4\mathcal{X}_T,
\end{cases}\quad t\in[0,T]
\end{equation} 
is guaranteed for every choice of the process $\w$. 

\subsection{The mean-field approximation of the optimal control problem for the regulator}\label{subsection_mfg_regulator}
We describe the mean-field approximation of optimal control problem that has to be solved by the regulator. In the mean-field setting, the regulator has to deal with a continuum of companies, as we discussed in Section \ref{subsection: mf companies game}. Defining $\overline{E}_t := \mathbb{E}[E_t|\ItF^0_t]$, the target function to be minimised by the regulator in this case is
\begin{equation*}
\mathcal{J}^0(\beta) := \mathbb{E} \bigg[ \int_0^TF( \bar{\mathcal{X}}_t,  \bar{\mathcal{Y}}_t; \beta_t) dt + G(  \bar{\mathcal{X}}_T ,\bar{\mathcal{Y}}_T; \beta_T)\bigg],
\end{equation*} 
where $\beta\in \mathbb{H}^2(\F^{W^0,\w};\R)$, $\w$ is a stochastic process defined on $(\Omega,\ItF, \prob,\F)$. With respect to the cost function defined in Section \ref{subsect_model_N:regulator}, we substitute the empirical mean of the emission of every company with the conditional expectation with respect the filtration generated by the common noise of emission of the representative company. We aim at re-writing the target function in terms of the state variable of the representative company $\mathcal{X}$ and the control of the regulator, with no explicit dependence on the emission process. 
\\
To this end, we notice that the empirical mean of emission of the companies is defined by 
\begin{equation*}
\mathfrak{m}^{(N)}(\underline{E}_t) :=\frac{1}{N} \sum_{j = 1}^N \int_0^t\alpha^{x,j}_sds - \mathfrak{m}^{(N)}(\underline{X}_t) + \mathfrak{m}^{(N)}(\underline{A}^i_0),\quad t\in[0,T].
\end{equation*}
Furthermore, as discussed in Remark \ref{rmk: controls reg N}, market clearing condition \eqref{eqn: eq price process N} is incoherent with the usual assumption made in the literature, i.e. $\w$ is adapted to $\F^0$. However, in the mean-field limit it is natural to suppose that the idiosyncratic noises cancel out, obtaining
\begin{equation}
\label{eqn: mf market clearing} 
-\beta_s := \lim_{N\to\infty} \frac{1}{N}\sum_{j = 1}^N\hat\alpha^{x,j}_s \approx \E\big[\hat{\alpha}^{x}_s |\ItF^0_s\big],\quad s\in[0,T],\quad \prob-a.s.,
\end{equation}
where $\hat\alpha^x$ is the optimal permit-trading rate of a representative company. We conclude that the number of permits emitted by the regulator should be (in the mean-field limit) independent of the idiosyncratic noises of the representative company ($W$) and adapted to $\F^0$. 

Adopting the approach developed in \cite{fujii2022mean}, we consider the mean-field limit of normalised version of the market clearing condition defined for the finite dimensional market in \eqref{eqn: eq price process N}:
\begin{equation}
\label{eqn: mf eq price}
\wmf_t :=  - \overline{Y}_t + 2c_{2,x} \beta_t,
\end{equation}
where, for every $t \in [0,T]$, $\overline{Y}_t := \E[Y_t\mid \ItF^0_t]$ and $\beta_t$ is strategy of the regulator. As a consequence, analogously to \cite[Section 6]{del2024mean}, we introduce the notation $\overline{\mathcal{Y}}_t := (\E[V_t|\ItF^0_t], \E[Y_t|\ItF^0_t])^\top$ for $t\in[0,T]$. In particular, imposing that the price process follows \eqref{eqn: mf eq price}, together with Assumption \ref{ass: price taker companies}, we obtain that the solution to the mean-field game for the representative company is described by the following FBSDE system (the adjoint process is defined as the mean-field version of \eqref{eqn: adjoint process N}):
\begin{equation}
\label{eqn: eq mf FBSDE company} 
\begin{cases}
d\mathcal{X}_t = \big[ A_1 \mathcal{X}_t + A_2( -Q_3^{-1} ( q_3  + Q_5\beta_t+ Q_6\overline{\mathcal{Y}}_t+ \frac{1}{2}A_2^\top \mathcal{Y}_t))\big] dt + \sigma(\mathcal{X}_t)dW_t + A_5dW^0_t,\\
d\mathcal{Y}_t = -\big( A_1^\top \mathcal{Y}_t + (A^K_3)^\top Z^{K}_t + 2Q_1\mathcal{X}_t + 2Q_2 \bar{\mathcal{X}}_t) + 2q_1 \big)dt + Z_t dWt + Z^{0}_t dW^0_t ,\\
\mathcal{X}_0 = (K_0,X_0)^\top,\\
\mathcal{Y}_T = Q_4\mathcal{X}_T.
\end{cases}
\end{equation}
We should prove the existence and uniqueness of the solution of \eqref{eqn: eq mf FBSDE company} for every given process $\beta\in\mathbb{H}^2(\F^0;\R)$. For the existence and uniqueness of solution of FBSDEs of the Mckean-Vlasov type (i.e. whose coefficients depend on the distribution of the solution itself), we base on \cite{carmona2013mean,carmona2015forward}.

\subsection{The passage to the limit}\label{subsect_MFG_passage_to_limit}
By Yamada-Watanabe representation theorem \cite[Theorem 1.33]{carmona2018probabilistic2}, there exists a deterministic progressively measurable function $\Psi$ such that
\begin{align*}
\Psi:\mathcal{D}([0,T],\R) \times \mathcal{C}([0,T],\R^2)&\to \mathcal{C}([0,T], \R^5)\times  \mathcal{D}([0,T],\R^2) \\
(\beta, W^0, W) & \mapsto\Psi( \beta, W^0,W):= (\mathcal{X}, \mathcal{Y}, Z^0,Z),
\end{align*} 
where $(\mathcal{X},\mathcal{Y},Z^0,Z)$ solves \eqref{eqn: eq mf FBSDE company}. Reasoning analogously, there exists a progressively measurable function $\Psi^N$, such that denoted by $\mathcal{K}^{i,N}:= (\mathcal{X}^i, \mathcal{Y}^i, Z^{0,i}, Z^i)$ the solution of \eqref{eqn: eq FBSDE company i}, for $i = 1,\dots,N$, the following relation holds
\begin{align*} 
\Psi^N: \mathcal{D}([0,T],\R) \times \mathcal{C}([0,T],\R^{N+1})&\to (\mathcal{C}([0,T],\R^4))^N\\ 
 ( \beta, W^0,W^1,\dots,W^N)& \mapsto \Psi^N(\mathcal{K}^{1,N},\dots, \mathcal{K}^{N,N}).
\end{align*}
We conclude that the family $(\mathcal{K}^{1,N},\dots, \mathcal{K}^{N,N})$ is exchangeable. As a consequence, the families 
\begin{equation*}
((K^1)^2,\dots, (K^N)^2),\quad ((Y^1)^2,\dots, (Y^N)^2),\quad ((V^1)^2,\dots, (V^N)^2),\quad(V^1Y^1,\dots, V^N,Y^N)
\end{equation*}
are exchangeable too. In analogy to Section \ref{subsection: mf companies game}, we conclude that
\begin{align*} 
\lim_{N\to \infty} \mathfrak{m}^{(2,N)}(\underline{L}_t) &= \E[L_t^2 \mid \ItF^0_t],\quad t\in[0,T],\quad \prob-a.s,\quad \forall L\in \{K,X,Y,V\},\\
\lim_{N\to \infty} \avg (V^j_tY^j_t)  &= \E[V_tY_t \mid \ItF^0_t],\quad t\in[0,T],\quad \prob-a.s,
\end{align*}
where $(\mathcal{X}, \mathcal{Y}) := ((K,X), (V,Y))$ solves \eqref{eqn: eq mf FBSDE company}. In the mean-field limit the optimal control problem to be solved by the regulator becomes
\begin{align}
    \label{eqn: mf target function regulator}
    &\mathcal{J}^0(\beta) \\
    &\nonumber\qquad :=\E \Bigg[ \int_0^T\Bigg[a(t) \Bigg(  -aA\E[K_t \mid \ItF^0_t] + b(1-\gamma)  A^2\E[ K^2_t \mid \ItF^0_t] + b \gamma A^2 \Big( \E[K_t\mid \ItF^0_t] \Big)^2  \\
    &\nonumber\qquad \qquad+ \gamma_1 \E[Y^2_t\mid\ItF^0_t] + \gamma_2 \Big(\E[Y_t \mid \ItF^0_t]\Big)^2 + \gamma_3 \E[ V^2_t\mid \ItF^0_t]  + \gamma_4 \E[V_t Y_t \mid \ItF^0_t] + \gamma_5\Bigg) \\
    &\nonumber\qquad\qquad  + \beta_t \big( a(t)\E[Y_t \mid \ItF^0_t] - 2(A^0-\theta)\big) +a(t)(b(t)-c_{2,x}) \beta^2_t \Bigg] dt + a(T) \lambda \E[X^2_T \mid \ItF^0_t] + B^2_T  \\
    &\nonumber\qquad\qquad + (\E[X_T\mid \ItF^0_T])^2+ (A^0 -\theta)^2 + 2 B_T \E[X_T \mid \ItF^0_T] - 2 (A^0 - \theta) \E[X_T \mid \ItF^0_T]\Bigg]\\
    &\nonumber\qquad =\E \Bigg[ \int_0^T\Bigg[a(t) \Bigg(  -aAK_t + b(1-\gamma) A^2 K^2_t + b  \gamma A^2\Big( \E[K_t\mid \ItF^0_t] \Big)^2 + \gamma_1Y^2_t+ \gamma_2 \Big(\E[Y_t \mid \ItF^0_t]\Big)^2 \\
    &\nonumber\qquad \qquad +\gamma_3 V^2_t + \gamma_4 V_t Y_t  + \gamma_5\Bigg) + \beta_t \big( a(t)\E[Y_t\mid \ItF^0_t]  - 2(A^0 - \theta)\big) +a(t)(b(t)-c_{2,x}) \beta^2_t \Bigg] dt \\
    &\nonumber \qquad\qquad+ a(T) \lambda X^2_T+B^2_T + (\E[X_T\mid \ItF^0_T])^2 + 2B_TX_T - 2(A^0-\theta)X_T\Bigg] + (A^0-\theta)^2\\
    &\nonumber\qquad = \E\Bigg[ \int_0^T (\langle Q^{0;X}_s \mathcal{X}_s,\mathcal{X}_s\rangle+\langle Q^{0;Y}_s \mathcal{Y}_s,\mathcal{Y}_s\rangle + \langle \bar{Q}^{0;X}_s \bar{\mathcal{X}}_s, \bar{\mathcal{X}}_s\rangle+ \langle \bar{Q}^{0;Y}_s \bar{\mathcal{Y}}_s, \bar{\mathcal{Y}}_s\rangle +a(s)(b(s)- c_{2,x})\beta_s^2\\
    &\nonumber\qquad \qquad  -2(A^0-\theta)\beta_s+ 2 S^{0;Y}_s \bar{\mathcal{Y}}_s \beta_s  + 2 \langle q^{0;X}_s , \mathcal{X}_s\rangle + 2 \langle q^{0;Y}_s ,\mathcal{Y}_s\rangle + \gamma_5) ds + \langle H^0 \mathcal{X}_T , \mathcal{X}_T\rangle\\
    &\nonumber\qquad\qquad + \langle H^1\bar{\mathcal{X}}_T, \bar{\mathcal{X}}_T\rangle + 2B_T \langle h^1, \bar{\mathcal{X}}_T \rangle+ B^2_T + \langle h^0, \mathcal{X}_T\rangle\Bigg]+ (A^0 - \theta)^2,
\end{align} 
where 
\begin{alignat*}{4}
    \gamma_1 &= \frac{1}{4c_{2,x}} + \frac{(a^e)^2}{4c_{2,f}} + \frac{1}{4c_{2,e}}, \qquad&
    Q^{0;X}_s &= a(s) \begin{pmatrix} b(1-\gamma)A^2 & 0 \\ 0 & 0 \end{pmatrix}, \\[6pt]
    \gamma_2 &= -\frac{1}{4c_{2,x}}, \qquad&
    \bar{Q}^{0;X}_s &= a(s) \begin{pmatrix} b\gamma A^2 & 0 \\ 0 & 0 \end{pmatrix}, \\[6pt]
    \gamma_3 &= \frac{(a^f)^2}{4c_{2,f}} + \frac{(a^g)^2}{4c_{2,g}}, \qquad&
    Q^{0;Y}_s &= a(s) \begin{pmatrix} \gamma_3 & \frac{1}{2}\gamma_4 \\ \frac{1}{2}\gamma_4 & \gamma_1 \end{pmatrix}, \\[6pt]
    \gamma_4 &= -\frac{a^f a^e}{2c_{2,f}}, \qquad&
    \bar{Q}^{0;Y}_s &= a(s) \begin{pmatrix} 0 & 0 \\ 0 & \gamma_2 \end{pmatrix}, \\[6pt]
    \gamma_5 &= -\left[ \frac{(c_{1,f})^2}{4c_{2,f}} + \frac{(c_{1,g})^2}{4c_{2,g}} + \frac{(c_{1,e})^2}{4c_{2,e}} \right], \qquad&
    S^{0;Y}_s &= \begin{pmatrix} 0 & \frac{1}{2}a(s) \end{pmatrix}, \\[6pt]
    & &
    q^{0;X}_s &= \begin{pmatrix} -\frac{1}{2}a(s)aA \\ 0 \end{pmatrix}, \\[6pt]
    & &
    H^0 &= \begin{pmatrix} 0 & 0 \\ 0 & a(T)\lambda \end{pmatrix}, \\[6pt]
    & &
    H^1 &= \begin{pmatrix} 0 & 0 \\ 0 & 1 \end{pmatrix}, \\[6pt]
    & &
    h^0 &= \begin{pmatrix} 0 \\ -2(A^0 - \theta) \end{pmatrix}, \\[6pt]
    & &
    h^1 &= \begin{pmatrix} 0 \\ 1 \end{pmatrix}.
\end{alignat*}

The second identity comes from the tower property, Fubini's theorem together with the fact that $\beta\in \mathbb{H}^2(\F^0;\R)$. 

The regulator’s optimisation problem in the mean-field limit, given by \eqref{eqn: mf target function regulator}, is a linear–quadratic stochastic control problem of McKean–Vlasov type. In the following we omit the term $(A^0 - \theta)^2$, appearing in the cost functional, since it doesn't depend either on the control $\beta$ or the state variable $(\mathcal{X}, B, \mathcal{Y})$. The state variables satisfy a forward–backward system of stochastic differential equations driven by idiosyncratic and common sources of uncertainty, with the solution indicated by $(\mathcal{X}, B, \mathcal{Y}, Z^0, Z)$.

In conclusion, the mean-field formulation of the optimal control problem of the regulator is given by $\inf_{\beta \in \mathbb{H}^2(\F^0;[0,\infty))}\mathcal{J}^0(\beta)$, where
\begin{align}
&\label{eqn: convex target function mf regulator}
    \mathcal{J}^0(\beta) \\
    &\nonumber\qquad = \E\Bigg[ \int_0^T (\langle Q^{0;X}_s \mathcal{X}_s,\mathcal{X}_s\rangle+\langle Q^{0;Y}_s \mathcal{Y}_s,\mathcal{Y}_s\rangle + \langle \bar{Q}^{0;X}_s \bar{\mathcal{X}}_s, \bar{\mathcal{X}}_s\rangle+ \langle \bar{Q}^{0;Y}_s \bar{\mathcal{Y}}_s, \bar{\mathcal{Y}}_s\rangle +a(s)(b(s) -  c_{2,x})( \beta_s)^2\\
    &\nonumber\qquad \qquad  -2(A^0 - \theta)\beta_s+ 2 \langle S^{0;Y}_s \bar{\mathcal{Y}}_s , \beta_s \rangle + 2 \langle q^{0;X}_s , \mathcal{X}_s\rangle + 2 \langle q^{0;Y}_s ,\mathcal{Y}_s\rangle + \gamma_5) ds + \langle H^0 \mathcal{X}_T , \mathcal{X}_T\rangle\\
    &\nonumber\qquad\qquad+ \langle h^0, \mathcal{X}_T\rangle+ \langle H^1\bar{\mathcal{X}}_T, \bar{\mathcal{X}}_T\rangle + 2B_T\langle h^1,\bar{\mathcal{X}}_T\rangle + B^2_T\Bigg],
\end{align}
subject to \eqref{eqn: eq mf FBSDE company}. 

\begin{remark}[On the mean-field formulation of the problem]
\label{rmk: defence of mfg}
One might wonder whether the model could be studied directly in a finite-dimensional setting, that is, for an economy populated by a finite number $N$ of firms. Although this is in principle possible, the analysis becomes considerably more intricate due to the impact of idiosyncratic noises on the equilibrium price and, consequently, on the agents’ strategies. 

In the $N$-player game, the equilibrium price $\w^N$ defined in \eqref{eqn: eq price process N} depends on the optimal controls of all firms and on the regulator’s policy. Since each firm’s optimal control reacts to its own idiosyncratic noise, the price $\w^N$ inherits a dependence on all firms’ individual shocks. As a result, the control of any given firm becomes correlated with the idiosyncratic noises of all other firms through a recursive structure that is extremely difficult to characterise. 

Passing to the mean-field limit resolves this difficulty. In the mean-field framework, the equilibrium price $\wmf$ becomes adapted to the common noise $\F^0$ only, as the aggregate effect of idiosyncratic noises cancels out in the limit. Consequently, each representative firm’s optimal control is adapted to its own idiosyncratic noise and to the common noise, but not to the idiosyncratic shocks of others. This greatly simplifies the structure of the equilibrium and enhances tractability.

A similar argument applies to the regulator’s problem. In the finite-dimensional case, the regulator realistically observes the common noise and the equilibrium price $\w^N$, which is itself correlated with the idiosyncratic noises of all firms. Therefore, the regulator’s control is adapted to the filtration $\mathcal{F}^{W^0, \w}$, where $\w$ depends recursively on the regulator’s own optimal policy. This feedback loop makes it extremely difficult to obtain explicit characterisations of the equilibrium. In the mean-field limit, however, the equilibrium price depends only on the common noise, so the regulator’s control becomes adapted to $\F^0$ alone, restoring analytical tractability.

Finally, it is worth noting that the mean-field formulation naturally accommodates a \emph{multi-population} extension, in which firms belong to different groups with heterogeneous technological or environmental characteristics, for instance, industries that rely more or less heavily on emissions to sustain their business models. In such a setting, one may reasonably assume that each population faces its own representative control problem within the same aggregate economy. As long as populations do not share additional sources of common information inaccessible to the others, the mean-field limit ensures that idiosyncratic noises continue to average out at the population level. Hence, the analytical simplifications afforded by the mean-field framework remain valid, and the extension to multiple populations can be obtained straightforwardly.
\end{remark}

\section{Analysis of MF optimal control problem subject to controlled FBSDE}\label{section_solution_MF_regulator_cpo}

The aim of this section is to study the solutions of the optimal control problem defined in Equations \eqref{eqn: eq mf FBSDE company} and \eqref{eqn: convex target function mf regulator}. This is a partial-information linear quadratic optimal control problem, where the state variable solves a controlled FBSDE system. We notice that the target function can be rewritten as 
\begin{align*}
    \mathcal{J}^0(\beta)    &= \bar{\mathcal{J}}^0(\beta) + \E\bigg[\int_0^T\bigg( \langle Q^{0,X}(\mathcal{X}_s - \bar{\mathcal{X}}_s), \mathcal{X}_s - \bar{\mathcal{X}}_s\rangle + \langle Q^{0,Y}(\mathcal{Y}_s - \bar{\mathcal{Y}}_s), \mathcal{Y}_s - \bar{\mathcal{Y}}_s\rangle \bigg) ds \\
    &\qquad+ \langle H^0 ( \mathcal{X}_T - \bar{\mathcal{X}}_T), \mathcal{X}_T - \bar{\mathcal{X}}_T\rangle \bigg],
\end{align*}
where 
\begin{align}
    &\label{eqn: convex target function mf regulator projected}
    \bar{\mathcal{J}}^0(\beta) \\
    &\nonumber\qquad = \E\Bigg[ \int_0^T \bigg(\langle (Q^{0;X}_s + \bar{Q}^{0;X}_s) \bar{\mathcal{X}}_s,\bar{\mathcal{X}}_s\rangle+\langle (Q^{0;Y}_s + \bar{Q}^{0;Y}_s) \bar{\mathcal{Y}}_s,\bar{\mathcal{Y}}_s\rangle +a(t)(b(s) -  c_{2,x})( \beta_s)^2\\
    &\nonumber\qquad \qquad-2(A^0 - \theta)\beta_s  + 2  S^{0;Y}_s \bar{\mathcal{Y}}_s  \beta_s  + 2 \langle q^{0;X}_s , \bar{\mathcal{X}}_s\rangle + 2 \langle q^{0;Y}_s ,\bar{\mathcal{Y}}_s\rangle + \gamma_5\bigg) ds + \langle H^0  \bar{\mathcal{X}}_T , \bar{\mathcal{X}}_T\rangle \\
    &\nonumber\qquad\qquad + \langle h^0, \bar{\mathcal{X}}_T\rangle+ \langle H^1\bar{\mathcal{X}}_T, \bar{\mathcal{X}}_T\rangle + 2B_T \langle h^1, \bar{\mathcal{X}}_T\rangle + B^2_T\rangle\Bigg].
\end{align}
We prove now the following result
\begin{lemma}
\label{lemma: filtering error independent of control}
\begin{equation*}\E\bigg[\int_0^T\bigg( \langle Q^{0,X}(\mathcal{X}_s - \bar{\mathcal{X}}_s), \mathcal{X}_s - \bar{\mathcal{X}}_s\rangle + \langle Q^{0,Y}(\mathcal{Y}_s - \bar{\mathcal{Y}}_s), \mathcal{Y}_s - \bar{\mathcal{Y}}_s\rangle \bigg) ds + \langle H^0  (\mathcal{X}_T - \bar{\mathcal{X}}_T), \mathcal{X}_T - \bar{\mathcal{X}}_T\rangle \bigg]
\end{equation*}
does not depend on $\beta$.
\end{lemma}
\begin{proof}
We recall that the solution $(\mathcal{X},\mathcal{Y},Z^0,Z)$ of the FBSDE \eqref{eqn: eq mf FBSDE company} is adapted to $\F$ since $\beta\in \mathbb{H}^2(\F^0;[0,\infty))$. In addition, let us notice that $\ItF^0_t = \ItF^0_s \vee \sigma\{W^0_u - W^0_s:\ u\in[0,t]\}$, for every $s\in[0,t]$, for every $t\in [0,T]$. Hence, for $s\leq t$, $\E[ \mathcal{X}_s|\ItF^0_t] = \E[ \mathcal{X}_s|\ItF^0_s] = \bar{\mathcal{X}}_s$ and the same holds for $\mathcal{Y}$, $Z^0$ and $Z$. In particular, we have
\begin{align*}
    \bar{\mathcal{X}}_t &= \mathcal{X}^0+ \E\bigg[\int_0^t \big( A_1 \mathcal{X}_s + A_2( -Q_3^{-1} ( q_3  + Q_5\beta_s+ Q_6\overline{\mathcal{Y}}_s+\frac{1}{2} A_2^\top \mathcal{Y}_s)) \big) ds \mid \ItF^0_t\bigg]+ A_5W^0_t\\
        &=  \mathcal{X}^0 + \int_0^t \big( A_1 \bar{\mathcal{X}}_s + A_2( -Q_3^{-1} ( q_3  + Q_5\beta_s+ (Q_6+\frac{1}{2}A_2^\top) \bar{\mathcal{Y}}_s)) \big) ds+ A_5W^0_t
\end{align*}
Reasoning similarly to \cite[Theorem 2.2]{wang2018introduction}, which is based on \cite[Theorem 8.1]{liptser2013statistics}, we have that
\begin{align*}
\bar{\mathcal{Y}}_t &=\mathcal{Y}_0 -\E\bigg[\int_0^t\big( A_1^\top \mathcal{Y}_s + (A^K_3)^\top Z^{K}_t + 2Q_1\mathcal{X}_s + Q_2\overline{\mathcal{X}}_s + 2q_1\big) ds - \int_0^tZ_s dW_s\\
&\qquad - \int_0^tZ^{0}_s dW^0_s\biggm| \ItF^0_t\bigg]\\
&= \mathcal{Y}_0- \int_0^t\big( A_1^\top \bar{\mathcal{Y}}_s + (A^K_3)^\top \bar{Z}^{K}_t + 2(Q_1+Q_2+q_1)\bar{\mathcal{X}}_s \big) ds + \int_0^t\bar{Z}^0_s dW^0_s.
\end{align*}
In conclusion, $(\bar{\mathcal{X}}, \bar{\mathcal{Y}}, \bar{Z}, \bar{Z}^0)$ solves the following FBSDE system 
\begin{equation}
\label{eqn: eq mf FBSDE company projected} 
\begin{cases}
    d\bar{\mathcal{X}}_t&=   \big[ A_1 \bar{\mathcal{X}}_t + A_2\big( -Q_3^{-1} ( q_3  + Q_5\beta_t+( Q_6 +\frac{1}{2} A_2^\top)\overline{\mathcal{Y}}_t)\big)\big] dt  + A_5dW^0_t,\\
    d\bar{\mathcal{Y}}_t &= - \big[A^\top_1 \bar{\mathcal{Y}}_t+ (A^K_3)^\top Z^{K}_t + 2(Q_1 + Q_2 +q_1) \bar{\mathcal{X}}_t \big] dt+ \bar{Z}^0_t dW^0_t, \\
\bar{\mathcal{X}}_0 &= (K_0,X_0)^\top,\\
\bar{\mathcal{Y}}_T &= Q_4\bar{\mathcal{X}}_T.
\end{cases}
\end{equation}

It is sufficient now to show that the filtering error,
\begin{equation*}
(\hat{\mathcal{X}} , \hat{\mathcal{Y}}, \hat{Z}) := 
(\mathcal{X} - \bar{\mathcal{X}}, \mathcal{Y} - \bar{\mathcal{Y}}, Z - \bar{Z}),
\end{equation*}
does not depend on $\beta$. We notice that $(\hat{\mathcal{X}} , \hat{\mathcal{Y}})$ solve the following FBSDE system
\begin{equation*}
    \begin{cases}
d\hat{\mathcal{X}}_t&=   \big( A_1 \hat{\mathcal{X}}_t - \frac{1}{2}A_2Q_3^{-1}  A_2^\top\hat{\mathcal{Y}}_t\big) dt  + A^K_3\hat{\mathcal{X}}_tdW_t,\\
d\hat{\mathcal{Y}}_t &= - (A^\top_1 \hat{\mathcal{Y}}_t+ (A^K_3)^\top \hat{Z}^{K}_t+ 2Q_1 \hat{\mathcal{X}}_t )dt+ \hat{Z}_t dW_t+ \hat{Z}^0_t dW^0_t, \\
\hat{\mathcal{X}}_0 &= 0,\\
\hat{\mathcal{Y}}_T &= Q_4\hat{\mathcal{X}}_T.
    \end{cases}
\end{equation*}
If uniqueness of strong solutions holds, by Yamada-Watanabe representation theorem \cite[Theorem 1.33]{carmona2018probabilistic2}, $ (\hat{\mathcal{X}}, \hat{\mathcal{Y}}, \hat{Z}^0, \hat{Z})$ is uniquely determined by $(W^0,W)$ and, in addition, does not depend on $\beta$. 
\end{proof}
In conclusion, we can reduce the regulator's problem to $\inf_{\beta \in \mathbb{H}^2(\F^0; \R)}\bar{\mathcal{J}}^0(\beta)$, introduced in \eqref{eqn: convex target function mf regulator projected} subject to $(\bar{\mathcal{X}},\bar{\mathcal{Y}})$, solution of \eqref{eqn: eq mf FBSDE company projected}. 

\subsection{Stochastic maximum principle in the regulator framework}

We start by recalling the FBSDEs system that defines the dynamics for the state: 
\begin{equation}
\label{eqn: fbsde mf repr company}
\begin{cases}
d\mathcal{X}_t 
    &= [A_{\mathcal{X}} \mathcal{X}_t + A_{\mathcal{Y}} \mathcal{Y}_t + \overline{A}_{\mathcal{Y}} \overline{{\mathcal{Y}}}_t + A_\beta \beta_t + a_\mathcal{X}]\, dt + [C_{\mathcal{X}} \mathcal{X}_t + c_K]\,dW_t^{K} +c_X\,dW_t^{X} + c_0\,dW_t^{0}, \\
    dB_t &= \beta_t dt,\\
    d{\mathcal{Y}}_t &= [B_{\mathcal{X}} \mathcal{X}_t + \overline{B}_{\mathcal{X}} \overline{{\mathcal{X}}}_t + B_{\mathcal{Y}} {\mathcal{Y}}_t + B_Z {Z}^K_t + b_\mathcal{Y}]\,dt  + {Z}^K_t\,d{W}_t^{K} + {Z}^X_t\,dW_t^{X}+ {Z}^0_t\,dW_t^{0},\\
    {\mathcal{X}}_0 &= (K_0,X_0)^\top,\\
    {\mathcal{Y}}_T &= Q_4{\mathcal{X}}_T.
\end{cases}
\end{equation}
where
\begin{alignat*}{5}
    A_{\mathcal{X}} &= A_1, &\quad
    A_{\mathcal{Y}} &= -\tfrac{1}{2}A_2 Q_3^{-1} A_2^\top, &\quad
    \overline{A}_{\mathcal{Y}} &= -A_2Q_3^{-1}Q_6, &\quad
    A_\beta &= -A_2Q_3^{-1}Q_5, &\quad
    a_{\mathcal{X}} &= -A_2 Q_3^{-1}q_3, \\[6pt]
    B_{\mathcal{X}} &= -2Q_1, &\quad
    \overline{B}_{\mathcal{X}} &= -2Q_2, &\quad
    B_{\mathcal{Y}} &= -A_1^\top, &\quad
    B_Z &= -A_3^K,&\quad
    b_{\mathcal{Y}} &= -2q_1, \\[6pt]
    C_{\mathcal{X}} &= A_3^K, &\quad
    c_K &= 0, &\quad
    c_X &= A_4^E, &\quad
    c_0 &= A_5.
\end{alignat*}
The state variable is now $(\bar{\mathcal{X}}_t, B_t)$. Nevertheless, because the process $B$ does not affect the dynamics of either $\bar{\mathcal{X}}$ or $\mathcal{Y}$, we consider $\bar{\mathcal{X}}$ and $B$ as distinct components.

The orthogonal decomposition of the states satisfies the following dynamics
\begin{align*}
    d\overline{\mathcal{X}}_t 
    &= [A_{\mathcal{X}} \overline{{\mathcal{X}}}_t + A_{\mathcal{Y}} \overline{\mathcal{Y}}_t + \overline{A}_{\mathcal{Y}} \overline{{\mathcal{Y}}}_t + A_\beta \beta_t + a_\mathcal{X}]\,dt + c_0\,dW_t^{0} \\
    &= [A_{\mathcal{X}} \overline{{\mathcal{X}}}_t + (A_{\mathcal{Y}} + \overline{A}_{\mathcal{Y}}) \overline{{\mathcal{Y}}}_t + A_\beta \beta_t + a_\mathcal{X}]\,dt + c_0\,dW_t^{0} \\
    d({\mathcal{X}} - \overline{{\mathcal{X}}})_t 
    &= [A_{\mathcal{X}}({\mathcal{X}} - \overline{{\mathcal{X}}})_t + A_{\mathcal{Y}}({\mathcal{Y}} - \overline{{\mathcal{Y}}})_t]\,dt + (C_{\mathcal{X}}({\mathcal{X}} - \overline{{\mathcal{X}}})_t + C_{\mathcal{X}} \overline{{\mathcal{X}}}_t +c_K )\,dW_t^{K} + c_X\,dW_t^{X} \\
    d\overline{\mathcal{Y}}_t 
    &= [(B_{\mathcal{X}} + \overline{B}_{\mathcal{X}}) \overline{\mathcal{X}}_t + B_{\mathcal{Y}} \overline{\mathcal{Y}}_t + B_Z \overline{Z}^K_t + b_\mathcal{Y}]\,dt + {Z}_t^{0}\,dW_t^{0} \\
    d({\mathcal{Y}} - \overline{{\mathcal{Y}}})_t 
    &= [B_{\mathcal{X}}({\mathcal{X}} - \overline{{\mathcal{X}}})_t + B_{\mathcal{Y}}({\mathcal{Y}} - \overline{{\mathcal{Y}}}) + B_Z({Z}^K_t - \overline{Z}^K_t)]\,dt + {Z}_t^{K}\,dW_t^{K} + {Z}_t^{X}\,dW^E_t
\end{align*}

Assuming the following ansatz for the form of the adjoint variable $\mathcal{Y}$
\begin{equation}\label{eq:ANSATZ_Y}
    {\mathcal{Y}}_t 
    = {P}_t ({\mathcal{X}}_t - \overline{{\mathcal{X}}}_t) + Q_t \overline{{\mathcal{X}}}_t +\Psi_t+ \varphi_t,
\end{equation}
with $P$, $Q$ and $\varphi$ deterministic function of time, such that $P_T=Q_4=Q_T$ and $\varphi_T=0$, and
\begin{align*}
    \Psi_t := \Psi_T-\int_t^T \psi_s ds + \int_t^T Z^\Psi_s dW^0_s,\quad \Psi_T = 0
\end{align*} 
with $\psi$ and $Z^\Psi$ stochastic processes adapted to $\mathbb{F}^0$, an application of the Ito formula yields
\begin{align*}
    d{\mathcal{Y}}_t 
    & = \left[ \dot{P}_t ({\mathcal{X}}_t - \overline{{\mathcal{X}}}_t) + {P}_t A_{\mathcal{X}} ({\mathcal{X}}_t - \overline{{\mathcal{X}}}_t) + {P}_t A_{\mathcal{Y}} ({\mathcal{Y}}_t - \overline{{\mathcal{Y}}}_t)\right]dt + \left[{P}_t C_{\mathcal{X}}({\mathcal{X}}_t - \overline{{\mathcal{X}}}_t) + {P}_t C_{\mathcal{X}} \overline{{\mathcal{X}}}_t +P_tc_K\right]dW_t^{K} \\
    & \quad + P_tc_X\,dW_t^{X} +\left[ \dot{Q}_t\overline{{\mathcal{X}}}_t + Q_t A_{\mathcal{X}} \overline{{\mathcal{X}}}_t + Q_t (A_{\mathcal{Y}} + \overline{A}_{\mathcal{Y}})\overline{{\mathcal{Y}}}_t + Q_t A_\beta \beta_t + Q_t a_\mathcal{X}\right]dt \\
    & \quad + (Q_t c_{0}-Z^\Psi_t)\,dW_t^{0} +\psi_t dt+ \dot{\varphi}_t\,dt.
\end{align*}
From the ansatz in \eqref{eq:ANSATZ_Y}, we derive 
\begin{equation}\label{eq:ansatz_cons}
    \overline{{\mathcal{Y}}}_t 
    = Q_t \overline{{\mathcal{X}}}_t + \Psi_t + \varphi_t
    \qquad 
    {\mathcal{Y}}_t - \overline{{\mathcal{Y}}}_t 
    = {P}_t({\mathcal{X}}_t - \overline{{\mathcal{X}}}_t)
\end{equation}
and so, replacing these expression in the previous equation, we get:

\begin{align}\label{eq:iniz_coef_dY}
    d{\mathcal{Y}}_t & = \left[ \dot{P}_t ({\mathcal{X}}_t - \overline{{\mathcal{X}}}_t) + {P}_t A_{\mathcal{X}} ({\mathcal{X}}_t - \overline{{\mathcal{X}}}_t) + {P}_t A_{\mathcal{Y}} {P}_t({\mathcal{X}}_t - \overline{{\mathcal{X}}}_t)\right]\,dt \\
    \nonumber
    & \quad + \left[{P}_t C_{\mathcal{X}}({\mathcal{X}}_t - \overline{{\mathcal{X}}}_t) + {P}_t C_{\mathcal{X}} \overline{{\mathcal{X}}}_t +P_tc_K\right]\,dW_t^{K} +P_tc_X\,dW_t^{X} \\
    \nonumber
    &\quad +\left[ \dot{Q}_t\overline{{\mathcal{X}}}_t + Q_t A_{\mathcal{X}} \overline{{\mathcal{X}}}_t + Q_t (A_{\mathcal{Y}} + \overline{A}_{\mathcal{Y}})\left[Q_t \overline{{\mathcal{X}}}_t + \Psi_t + \varphi_t\right]+ Q_t A_\beta \beta_t + Q_t a_\mathcal{X}\right]\,dt \\
    \nonumber
    &\quad + (Q_t c_{0}-Z^\Psi_t)\,dW_t^{0} +\psi_t dt+ \dot{\varphi}_t\,dt\\
    \nonumber
    & = \Big\{ \left[\dot{{P}}_t +{P}_t A_{\mathcal{Y}} {P}_t + {P}_t A_{\mathcal{X}}\right]({\mathcal{X}}_t - \overline{{\mathcal{X}}}_t) + \left[ \dot{Q}_t + Q_t (A_{\mathcal{Y}}+ \overline{A}_{\mathcal{Y}}) Q_t + Q_t A_{\mathcal{X}} \right] \overline{{\mathcal{X}}}_t \\ 
    \nonumber
    &\quad  +\left[ \psi_t + Q_t A_\beta \beta_t + Q_t (A_{\mathcal{Y}}+ \overline{A}_{\mathcal{Y}})\Psi_t \right] + \left[\dot{\varphi}_t +Q_t a_\mathcal{X} + Q_t (A_{\mathcal{Y}}+ \overline{A}_{\mathcal{Y}})\varphi_t\right] \Big\}\,dt\\  
    \nonumber
    & \quad + \left[{P}_t C_{\mathcal{X}}({\mathcal{X}}_t - \overline{{\mathcal{X}}}_t) + {P}_t C_{\mathcal{X}} \overline{{\mathcal{X}}}_t +P_tc_K\right]\,dW_t^{K} +P_tc_X\,dW_t^{X} + (Q_t c_{0}-Z^\Psi_t)\,dW_t^{0}
\end{align}

Going back to the original expression for $\mathcal{Y}$ written in a convenient form (i.e. according to the orthogonal decomposition) and then exploiting the identities in Equation \eqref{eq:ansatz_cons}, we obtain
\begin{align*}
    d{\mathcal{Y}}_t 
    & = [B_{\mathcal{X}} ({\mathcal{X}}_t - \overline{{\mathcal{X}}}_t) + (B_{\mathcal{X}}+\overline{B}_{\mathcal{X}}) \overline{{\mathcal{X}}}_t + B_{\mathcal{Y}} ({\mathcal{Y}}_t - \overline{\mathcal{Y}}_t) + B_{\mathcal{Y}}  \overline{\mathcal{Y}}_t + B_Z ({Z}^K_t-\overline{Z}^K_t) + B_Z \overline{Z}^K_t + b_\mathcal{Y}]\,dt \\
    & \quad + {Z}^K_t\,d{W}_t^{K} + {Z}^X_t\,dW_t^{X}+ {Z}^0_t\,dW_t^{0}\\
    & = [B_{\mathcal{X}} ({\mathcal{X}}_t - \overline{{\mathcal{X}}}_t) + (B_{\mathcal{X}}+\overline{B}_{\mathcal{X}}) \overline{{\mathcal{X}}}_t + B_{\mathcal{Y}} ({\mathcal{Y}}_t - \overline{\mathcal{Y}}_t) + B_{\mathcal{Y}}  \overline{\mathcal{Y}}_t + B_Z ({Z}^K_t-\overline{Z}^K_t) + B_Z \overline{Z}^K_t + b_\mathcal{Y}]\,dt \\
    & \quad + {Z}^K_t\,d{W}_t^{K} + {Z}^X_t\,dW_t^{X}+ {Z}^0_t\,dW_t^{0}\\
    & = \Big\{B_{\mathcal{X}} ({\mathcal{X}}_t - \overline{{\mathcal{X}}}_t) + (B_{\mathcal{X}}+\overline{B}_{\mathcal{X}}) \overline{{\mathcal{X}}}_t + B_{\mathcal{Y}} {P}_t({\mathcal{X}}_t - \overline{{\mathcal{X}}}_t) + B_{\mathcal{Y}}  [Q_t \overline{{\mathcal{X}}}_t + \Psi_t + \varphi_t] \\
    & \quad + B_Z ({Z}^K_t-\overline{Z}^K_t) + B_Z \overline{Z}^K_t + b_\mathcal{Y}\Big\}\,dt + {Z}^K_t\,d{W}_t^{K} + {Z}^X_t\,dW_t^{X}+ {Z}^0_t\,dW_t^{0}.
\end{align*}
We start by matching the volatility terms and we get
\begin{align}
    {Z}^K_t = \left[{P}_t C_{\mathcal{X}}({\mathcal{X}}_t - \overline{{\mathcal{X}}}_t) + {P}_t C_{\mathcal{X}} \overline{{\mathcal{X}}}_t +P_tc_K\right],
    \quad
    {Z}^X_t = P_tc_X,
    \quad
    {Z}^0_t = (Q_t c_{0}-Z^\Psi_t).
\end{align}
Then, we can exploit the expression for  ${Z}^K_t$ and write
\begin{align}\label{eq:fin_coef_dY}
    d{\mathcal{Y}}_t 
    & = \Big\{(B_{\mathcal{X}}+B_{\mathcal{Y}} {P}_t) ({\mathcal{X}}_t - \overline{{\mathcal{X}}}_t) + (B_{\mathcal{X}}+\overline{B}_{\mathcal{X}}+   B_{\mathcal{Y}} Q_t) \overline{{\mathcal{X}}}_t +   B_{\mathcal{Y}}  \Psi_t +  B_{\mathcal{Y}}\varphi_t \\ 
    &\nonumber \quad + B_Z ({Z}^K_t-\overline{Z}^K_t) + B_Z \overline{Z}^K_t + b_\mathcal{Y}\Big\}\,dt + {Z}^K_t\,d{W}_t^{K} + {Z}^X_t\,dW_t^{X}+ {Z}^0_t\,dW_t^{0}\\  
    & \nonumber= \Big\{(B_{\mathcal{X}}+B_{\mathcal{Y}} {P}_t) ({\mathcal{X}}_t - \overline{{\mathcal{X}}}_t) + (B_{\mathcal{X}}+\overline{B}_{\mathcal{X}}+   B_{\mathcal{Y}} Q_t) \overline{{\mathcal{X}}}_t +   B_{\mathcal{Y}} \Psi_t + B_{\mathcal{Y}}\varphi_t \\  
    &\nonumber \quad + B_Z {P}_t C_{\mathcal{X}}({\mathcal{X}}_t - \overline{{\mathcal{X}}}_t) + B_Z {P}_t C_{\mathcal{X}} \overline{{\mathcal{X}}}_t + B_Z P_tc_K + b_\mathcal{Y}\Big\}\,dt + {Z}^K_t\,d{W}_t^{K} + {Z}^X_t\,dW_t^{X}+ {Z}^0_t\,dW_t^{0}\\
    &\nonumber = \Big\{(B_{\mathcal{X}}+B_{\mathcal{Y}} {P}_t  + B_Z {P}_t C_{\mathcal{X}}) ({\mathcal{X}}_t - \overline{{\mathcal{X}}}_t) + (B_{\mathcal{X}}+\overline{B}_{\mathcal{X}}+   B_{\mathcal{Y}} Q_t + B_Z {P}_t C_{\mathcal{X}} ) \overline{{\mathcal{X}}}_t + B_{\mathcal{Y}}  \Psi_t + B_{\mathcal{Y}}\varphi_t \\  
    &\nonumber \quad + B_Z P_tc_K + b_\mathcal{Y}\Big\}\,dt + {Z}^K_t\,d{W}_t^{K} + {Z}^X_t\,dW_t^{X}+ {Z}^0_t\,dW_t^{0}.
\end{align}
Finally, matching term by term the components in the drift resp. in \eqref{eq:fin_coef_dY} and \eqref{eq:iniz_coef_dY}, we get
\begin{alignat}{2}
    & \dot{{P}}_t +{P}_t A_{\mathcal{Y}} {P}_t + {P}_t A_{\mathcal{X}}  - B_{\mathcal{Y}} {P}_t  - B_Z {P}_t C_{\mathcal{X}} - B_{\mathcal{X}} = 0,&& {P_T = Q_4}\\
    & \dot{Q}_t + Q_t (A_{\mathcal{Y}}+ \overline{A}_{\mathcal{Y}}) Q_t + Q_t A_{\mathcal{X}} - B_Z {P}_t C_{\mathcal{X}} - B_{\mathcal{X}} - \overline{B}_{\mathcal{X}} -   B_{\mathcal{Y}} Q_t  = 0,\qquad&& {Q_T = Q_4}\\
    &\label{eqn: psi} \psi_t + Q_t A_\beta \beta_t + [Q_t (A_{\mathcal{Y}}+ \overline{A}_{\mathcal{Y}})- B_{\mathcal{Y}}] {\Psi_t} =0,&& {\Psi_T = 0}\\
    & \dot{\varphi}_t  + [Q_t (A_{\mathcal{Y}}+ \overline{A}_{\mathcal{Y}}) - B_{\mathcal{Y}}]\varphi_t - B_Z P_tc_K + Q_t a_\mathcal{X} - b_\mathcal{Y} =0, &&{\varphi_T = 0}.
\end{alignat}
Finally, \eqref{eqn: psi}  implies that 
\begin{equation}
    d\Psi_t = - \bigg[Q_t A_\beta \beta_t + [Q_t (A_{\mathcal{Y}}+ \overline{A}_{\mathcal{Y}})- B_{\mathcal{Y}}] \Psi_t\bigg] dt - Z^\Psi_t dW^0_t, \quad {\Psi_T = 0}.
\end{equation}
For the moment let us assume that the solutions of the ODEs and the SDE for $\Psi$ exist uniquely. By Ansatz \eqref{eq:ANSATZ_Y}, we have for $t\in[0,T]$ 
\begin{equation*}
    \begin{cases}
        \bar{\mathcal{Y}}_t = Q_t \bar{\mathcal{X}}_t + \Psi_t + \varphi_t,\\
        \bar{Z}_t = ( P_t C_{\mathcal{X}} \bar{\mathcal{X}}_t + P_t c_K,\quad P_t c_X),\\
        \bar{Z}^0_t := Q_t c_0 {- Z^\Psi_t }.
    \end{cases}
\end{equation*}
We introduce the following notation
\begin{alignat*}{2}
    A^\mathcal{X}_t        &:= A_\mathcal{X} + (A_\mathcal{Y} + \bar{A}_\mathcal{Y})Q_t, \quad & 
    A^\Psi_t               &:= A_\mathcal{Y} + \bar{A}_\mathcal{Y} \\[3pt]
    A^\beta_t              &:= A_\beta, & 
    a^\mathcal{X}_t        &:= (A_\mathcal{Y} + \bar{A}_\mathcal{Y}) \varphi_t + a_\mathcal{X} \\[8pt]
    B^\Psi_t               &:= -Q_t(A_\mathcal{Y} + \bar{A}_\mathcal{Y}) + B_\mathcal{Y}, & 
    B^\beta_t              &:= -Q_tA_\beta \\[8pt]
    Q^\mathcal{X}_t        &:= Q^{0;X}_t + \bar{Q}^{0;X}_t + Q^\top_t(Q^{0;Y}_t + \bar{Q}^{0;Y}_t)Q_t,\quad & 
    Q^{\mathcal{X}\Psi}_t  &:= (Q^{0,Y}_t + \bar{Q}^{0;Y}_t) Q_t \\[3pt]
    Q^\Psi_t               &:= Q^{0,Y}_t + \bar{Q}^{0;Y}_t && \\[8pt]
    q^\mathcal{X}_t        &:= q^{0;X}_t + Q^\top_t q^{0;Y}_t + Q^\top_t (Q^{0;Y}_t + \bar{Q}^{0;Y}_t)\varphi_t, & 
    q^\Psi_t               &:= q^{0;Y}_t + (Q^{0;Y}_t + \bar{Q}^{0;Y}_t)\varphi_t \\[8pt]
    Q^\beta_t              &:= b(t) - c_{2,x}, & 
    q^{\beta\mathcal{X}}_t &:= 2S^{0;Y}_t Q_t \\[3pt]
    q^{\beta\Psi}_t        &:= 2S^{0;Y}_t, & 
    q^\beta_t              &:= 2[S^{0;Y}_t \varphi_t - (A_0 - \theta)]
\end{alignat*}
The state variable becomes
\begin{align*}
    d 
    \begin{pmatrix} 
    \bar{\mathcal{X}}_t\\
    B_t\\
    \Psi_t 
    \end{pmatrix} 
    &= \begin{pmatrix} 
    A_\mathcal{X} \bar{\mathcal{X}}_t + (A_\mathcal{Y} + \bar{A}_{\mathcal{Y}} ) (Q_t \bar{\mathcal{X}}_t + \Psi_t + \varphi_t) + A_\beta \beta_t + a_\mathcal{X} \\
    \beta_t\\ 
    - \{ Q_t A_\beta \beta_t + [Q_t( A_\mathcal{Y} + \bar{A}_{\mathcal{Y}})- B_\mathcal{Y}]\Psi_t \} 
    \end{pmatrix} dt + \begin{pmatrix}c_0 \\0\\ -Z^\Psi_t\end{pmatrix}dW^0_t\\
  &= \begin{pmatrix} 
    [A_\mathcal{X} + (A_\mathcal{Y} + \bar{A}_{\mathcal{Y}}) Q_t ] \bar{\mathcal{X}}_t + (A_\mathcal{Y} + \bar{A}_{\mathcal{Y}} )\Psi_t + A_\beta \beta_t + (A_\mathcal{Y} + \bar{A}_{\mathcal{Y}} )\varphi_t + a_\mathcal{X} \\
    \beta_t\\
     [-Q_t( A_\mathcal{Y} + \bar{A}_{\mathcal{Y}})+ B_\mathcal{Y}]\Psi_t - Q_t A_\beta \beta_t  
    \end{pmatrix} dt + \begin{pmatrix}c_0 \\0\\- Z^\Psi_t \end{pmatrix}dW^0_t\\\nonumber
    & = \begin{pmatrix} 
    A^{\mathcal{X}}_t \bar{\mathcal{X}}_t + A^{\Psi}_t\Psi_t + A^\beta_t \beta_t + a^\mathcal{X}_t \\
    \beta_t\\
     B^{\Psi}_t\Psi_t + B^{\beta}_t \beta_t  
    \end{pmatrix} dt + \begin{pmatrix}c_0 \\ 0 \\ -Z^\Psi_t \end{pmatrix}dW^0_t.
\end{align*}
Under the condition that $b(s) - c_{2,x}>0$ for every $s\in[0,T]$, the optimal control problem has a convex linear quadratic target function in state and control subject to a state variable $(\bar{\mathcal{X}}_t,B_t, \Psi_t)_{t\in[0,T]}$, which is the solution of a linear SDE. By \eqref{eqn: convex target function mf regulator projected}, the target function of the regulator then becomes: 
\begin{align}
    \bar{\mathcal{J}}^0(\beta)  
    & = \mathbb{E}\Bigg[ \int_0^T \bigg\{
    \bigg\langle  \big[(Q^{0;X}_s + \bar{Q}^{0;X}_s) + Q_s^\top (Q^{0;Y}_s + \bar{Q}^{0;Y}_s)Q_s\big]\bar{\mathcal{X}}_s, \bar{\mathcal{X}}_s \bigg\rangle \\
    & \nonumber \quad + 2\bigg\langle  (Q^{0;Y}_s + \bar{Q}^{0;Y}_s)Q_s\bar{\mathcal{X}}_s , \Psi_s\bigg\rangle+ \langle  (Q^{0;Y}_s + \bar{Q}^{0;Y}_s) \Psi_s ,\Psi_s \rangle \\
    & \nonumber \quad +2\langle  q^{0;X}_s +Q_s^\top q^{0;Y}_s+Q_s^T(Q^{0;Y}_s + \bar{Q}^{0;Y}_s) \varphi_s, \bar{\mathcal{X}}_s\rangle + 2\langle  q^{0;Y}_s +(Q^{0;Y}_s + \bar{Q}^{0;Y}_s) \varphi_s , \Psi_s\rangle \\
    & \nonumber \quad + (b(s) - c_{2,x} )\beta^2_s + 2\bigg[ S^{0;Y}_s (Q_s \bar{\mathcal{X}}_s + \Psi_s + \varphi_s)-(A^0-\theta)\bigg] \beta_s+ 2\langle  q^{0;Y}_s,\varphi_s \rangle  \\
    &\nonumber \quad + \langle  (Q^{0;Y}_s + \bar{Q}^{0;Y}_s) \varphi_s ,\varphi_s \rangle + \gamma_5 \bigg\}ds  + {\langle } (H^0 +H^1)\bar{\mathcal{X}}_T , \bar{\mathcal{X}}_T\rangle + \langle  2B_Th^1 + h^0, \bar{\mathcal{X}}_T\rangle + B^2_T\Bigg]\\
    & \nonumber =: \mathbb{E}\Bigg[ \int_0^T \bigg\{
    \bigg\langle  Q^{\mathcal{X}}_s\bar{\mathcal{X}}_s, \bar{\mathcal{X}}_s \bigg\rangle  + 2\bigg\langle  Q^{\mathcal{X }\Psi}_s\bar{\mathcal{X}}_s , \Psi_s\bigg\rangle+ \langle  Q^{\Psi}_s \Psi_s ,\Psi_s \rangle +2\langle  q^{\mathcal X}_s, \bar{\mathcal{X}}_s\rangle + 2\langle  q^{\Psi}_s, \Psi_s\rangle \\
    & \nonumber \quad + Q^\beta_s \beta^2_s + \bigg[ q^{\beta \mathcal{X}}_s \bar{\mathcal{X}}_s + q^{\beta \Psi}_s\Psi_s + q^{\beta}_s \bigg] \beta_s \bigg\}ds + {\langle } H \bar{\mathcal{X}}_T , \bar{\mathcal{X}}_T\rangle + \langle  2B_Th^1 + h^0, \bar{\mathcal{X}}_T\rangle + B^2_T\Bigg]  
    \\
    &\nonumber \quad  + \mathbb{E}\Bigg[ \int_0^T \bigg\{
    2\langle  q^{0;Y}_s,\varphi_s \rangle + \langle  (Q^{0;Y}_s + \bar{Q}^{0;Y}_s) \varphi_s ,\varphi_s \rangle + \gamma_5 \bigg\}ds \Bigg],
\end{align}
where we forget about the second expectation in the rest of the optimisation process as it is independent of states and controls.

We aim at solving the optimal control problem above. In order to apply the stochastic maximum principle, we introduce the adjoint process $(P^\mathcal{X}, P^B, R)$. Since the dynamics of the state variables is described by a system of FBSDEs, the corresponding/associated Hamiltonian has the following form 
\begin{align}
    &{H(t,\bar{\mathcal{X}}_t,B_t, \Psi_t, Z^{\Psi}_t,P^{\mathcal{X}}_t,P^B_t, R_t, Z^{\mathcal{X}}_t, Z^B_t, \beta_t)} \\ \nonumber
    \qquad & := \langle P^{\mathcal{X}}_t, A^{\mathcal{X}}_t \bar{\mathcal{X}}_t + A^{\Psi}_t \Psi_t + A^{\beta}_t \beta_t +a^{\mathcal{X}}_t\rangle + \langle P^B_t, \beta_t\rangle + \langle Z^{\mathcal{X}}_t, c_0\rangle + \langle R_t, B^\Psi_t \Psi_t +B^\beta_t \beta_t\rangle  \\ \nonumber
    &\quad + \langle Q^{\mathcal{X}}_t \bar{\mathcal{X}}_t, \bar{\mathcal{X}}_t\rangle+ 2 \langle Q^{\mathcal{X}\Psi}_t\bar{\mathcal{X}}_t, \Psi_t\rangle + \langle Q^{\Psi}_t\Psi_t, \Psi_t\rangle + 2\langle q^{\mathcal{X}}_t,\bar{\mathcal{X}}_t \rangle + 2\langle q^{\Psi}_t,\Psi_t \rangle + Q^\beta_t \beta_t^2\\ \nonumber
    &\quad + \bigg[ q^{\beta \mathcal{X}}_t \bar{\mathcal{X}}_t + q^{\beta \Psi}_t\Psi_t + q^{\beta}_t \bigg] \beta_t.
\end{align}
Thus, computing its partial derivatives we get
\begin{equation*}
\begin{cases}
    H_{\mathcal X} 
    & = (A^{\mathcal{X}}_t)^\top  P^\mathcal{X}_t + 2Q^{\mathcal{X}}_t \bar\X_t + 2(Q^{\X \Psi}_t)^\top \Psi_t +2q^\X_t + (q^{\beta \X}_t)^\top  \beta_t \\
    H_B
    & =0,\\
    H_\Psi 
    & = (A^{\Psi}_t)^\top  P^\mathcal{X}_t + (B^{\Psi}_t)^\top  R_t + 2Q^{\mathcal{X}\Psi}_t \bar\X_t + 2(Q^{ \Psi}_t)^\top\Psi_t +2q^\Psi_t + (q^{\beta \Psi}_t)^\top  \beta_t \\
    H_{Z^\Psi} 
    &= 0,\\
    H_{Z^\mathcal{X}} &= c_0,\\
    H_{Z^B} 
    &= 0,
\end{cases}
\end{equation*}
so that the dynamics for the adjoint processes read as
\begin{align*}
\begin{cases}
&dR_t = -\left((A^{\Psi}_t)^\top  P^\mathcal{X}_t + (B^{\Psi}_t)^\top  R_t + 2Q^{\mathcal{X}\Psi}_t \bar\X_t + 2(Q^{ \Psi}_t)^\top \Psi_t +2q^\Psi_t + (q^{\beta \Psi}_t)^\top  \beta_t\right) dt + 0 dW^0_t, \\
& R_0 = 0;\\
&dP^\mathcal{X}_t = -\left((A^{\mathcal{X}}_t)^\top  P^\mathcal{X}_t + 2Q^{\mathcal{X}}_t \bar\X_t + 2(Q^{\X \Psi}_t)^\top \Psi_t +2q^\X_t + (q^{\beta \X}_t)^\top  \beta_t\right) dt + Z^\mathcal{X}_t dW^0_t, \\
&P^\mathcal{X}_T = 2H \bar{\mathcal{X}}_T + 2B_T h^1 + h_0,\\
& dP^B_t = Z^B_t dW^0_t,\\
&P^B_T = 2B_T + 2(h^1)^\top\bar{\mathcal{X}}_T.
\end{cases}
\end{align*}

Now, if we differentiate w.r.t. $\beta$ we obtain the following 
\begin{align*}
  H_\beta = (A^\beta_t)^\top P^\mathcal{X}_t +P^B_t+ (B^\beta_t)^\top R_t + 2Q^\beta_t \beta_t + q^{\beta \X}_t \bar \X_t+ q^{\beta \Psi}_t \Psi_t + q^{\beta}_t
\end{align*}
and so by the FOC the candidate optimal control should have the following form:
\begin{align}
    \hat\beta_t
    = -\frac{1}{2Q^\beta_t} \left((A^\beta_t)^\top P^\mathcal{X}_t + P^B_t+ (B^\beta_t)^\top R_t  + q^{\beta \X}_t \bar \X_t+ q^{\beta \Psi}_t \Psi_t + q^{\beta}_t\right)
\end{align}

We introduce now the following notation, under which the coefficients of the FBSDEs can be handled more efficiently: 
\begin{align*}
\tilde{O}&:= A_{\mathcal{X}}= A_1 
= \begin{pmatrix}
    -\delta & 0\\
    0 & 0
\end{pmatrix}\\
\tilde{M}& := A_\mathcal{Y}+\overline{A}_\mathcal{Y} = -\frac{1}{2} A_2 Q_3^{-1} A_2^\top  - A_2 Q_3^{-1} Q_6 
= \frac{1}{2} \begin{pmatrix} 
\frac{(a^f)^2}{c_{2,f}}+\frac{(a^g)^2}{c_{2,g}}		&-\frac{a^e a^f}{c_{2,f}}\\
-\frac{a^e a^f}{c_{2,f}}	&\frac{(a^e)^2}{c_{2,f}} +\frac{1}{c_{2,e}}	
\end{pmatrix}
\\
A_\beta &= -A_2 Q_3^{-1} Q_5 
= \begin{pmatrix} 
0 \\
- 1	
\end{pmatrix}
\\
    a_\mathcal{X} &= -A_2 Q_3^{-1}q_3
    = -\frac{1}{2}\begin{pmatrix} 
\frac{c_{1,f}}{c_{2,f}}a^{f} + \frac{c_{1,g}}{c_{2,g}}a^{g} \\
-\frac{c_{1,f}}{c_{2,f}}a^{e} + \frac{c_{1,e}}{c_{2,e}}
\end{pmatrix}
\\
\tilde{N} 
& := \begin{pmatrix}
    0 & 0 \\
    0 & \frac{1}{2 (b(t)-c_{2,x})}
\end{pmatrix}
\\
 \tilde{T}&:= Q^{0;Y}_t + \overline{Q}^{0;Y}_t= a(t) \begin{pmatrix}
\gamma_2+\gamma_3&\frac{1}{2} \gamma_4 \\ \frac{1}{2}\gamma_4 &\gamma_1\end{pmatrix}
\\
\tilde{U}&:= Q^{0;X}_t + \overline{Q}^{0;X}_t
= a(t) \begin{pmatrix}
bA^2&0\\
0&0
\end{pmatrix}
\\
B_y &= -A_1^\top  = - \widetilde{O}
\\
\widetilde{F} 
& :=
\begin{pmatrix}
    0\\
    1
\end{pmatrix}
\end{align*}
Hence, the new system of FBSDEs has the following form:
\begin{align*}
    d\bar \X_t 
    & = \left(A^{\mathcal{X}}_t \bar{\mathcal{X}}_t + A^{\Psi}_t\Psi_t -\frac{1}{2Q^\beta_t} A^\beta_t \left((A^\beta_t)^\top P^\mathcal{X}_t + P^B_t + (B^\beta_t)^\top R_t  + q^{\beta \X}_t \bar \X_t+ q^{\beta \Psi}_t \Psi_t + q^{\beta}_t\right) + a^\mathcal{X}_t\right) dt \\
    &\qquad+ c_0 dW^0_t\\
    & = 
    \bigg\{[\tilde{O} + (\tilde{M} + a(t)\tilde{N}) Q_t] \bar{\mathcal{X}}_t + (\tilde{M} + a(t) \tilde{N})\Psi_t + \tilde{N}Q_t R_t - \tilde{N}P^\mathcal{X}_t +\tilde{N}\tilde{F} P^B_t \\
    &\qquad + \bigg[ (\tilde{M} + a(t) \tilde{N})\varphi_t + a_\mathcal{X} - 2 \tilde{N}\widetilde{F}(A^0 - \theta)\bigg]\bigg\} dt+ c_0 dW^0_t\\
    \bar \X_0
    & = \begin{pmatrix}
        K_0&
        X_0
    \end{pmatrix}^\top\\
     dB_t 
     &=  -\frac{1}{2Q^\beta_t} \left((A^\beta_t)^\top P^\mathcal{X}_t + P^B_t+ (B^\beta_t)^\top R_t  + q^{\beta \X}_t \bar \X_t+ q^{\beta \Psi}_t \Psi_t + q^{\beta}_t\right)dt\\
    & = \big\{\tilde{F}^T \tilde{N} P^\mathcal{X}_t -\frac{1}{2(b(t)-c_{2,x})} P^B_t - \tilde{F}^T \tilde{N} Q_t R_t - a(t)\tilde{F}^T \tilde{N} Q_t \bar{\mathcal{X}}_t - a(t)\tilde{F}^T \tilde{N}\Psi_t \\
    &\qquad - a(t) \tilde{F}^T \tilde{N}\varphi_t + \frac{A^0-\theta}{(b(t)-c_{2,x})}\big\} dt\\
    B_0&=0\\
    dR_t 
    & = \Bigg[-(A^{\Psi}_t)^\top P^\mathcal{X}_t - (B^{\Psi}_t)^\top R_t - 2Q^{\mathcal{X}\Psi}_t \bar\X_t - 2(Q^{ \Psi}_t)^\top\Psi_t -2q^\Psi_t  \\
    & \qquad +\frac{1}{2Q^\beta_t} (q^{\beta \Psi}_t)^\top \left((A^\beta_t)^\top P^\mathcal{X}_t + P^B_t+ (B^\beta_t)^\top R_t  + q^{\beta \X}_t \bar \X_t+ q^{\beta \Psi}_t \Psi_t + q^{\beta}_t\right) \Bigg] dt, \\
    &= \bigg\{( - 2\tilde{T} +a(t)^2 \tilde{N})Q_t \bar{\mathcal{X}}_t + [-2\tilde{T} + \tilde{N}a(t)^2] \Psi_t + [\tilde{O} + (\tilde{M} + a(t) \tilde{N})Q_t] R_t\\
    &\qquad -(\widetilde{M} + \widetilde{N}a(t))P^\mathcal{X}_t - a(t)\tilde{N}\tilde{F} P^B_t+ \bigg[(- 2\tilde{T}  + a(t)^2\tilde{N})\varphi_t- a(t) \tilde{N}\widetilde{F}[2(A^0-\theta)]\bigg]\bigg\} dt,\\
    R_0 & = (0 \quad 0)^\top, \\
    d\Psi_t 
    & = \left[B^{\Psi}_t\Psi_t  -\frac{1}{2Q^\beta_t} B^{\beta}_t \left((A^\beta_t)^\top  P^\mathcal{X}_t + P^B_t + (B^\beta_t)^\top  R_t  + q^{\beta \X}_t \bar \X_t+ q^{\beta \Psi}_t \Psi_t + q^{\beta}_t\right)  \right] dt +  Z^\Psi_t dW^0_t\\
    & = \bigg\{ - a(t) Q_t \widetilde{N} Q_t \bar{\mathcal{X}}_t - ( \widetilde{O}+Q_t ( \widetilde{M} + a(t) \widetilde{N}) )\Psi_t -  Q_t \widetilde{N} Q_tR_t + Q_t \widetilde{N}  P^\mathcal{X}_t \\
    &\qquad- a(t)\tilde{N}\tilde{F} P^B_t + Q_t \widetilde{N}\left[- a(t) \varphi_t + \widetilde{F} \right] - Q_t \widetilde{N}\left[ a(t) \varphi_t + 2(A^0-\theta)\widetilde{F} \right]\bigg\} dt + Z^\Psi_t dW^0_t,\\
    \Psi_T & = (0 \quad 0)^\top,\\
    dP^\mathcal{X}_t 
    & = \bigg\{[ -2 \widetilde{U} +Q_t [-2\widetilde{T}+a(t)^2\widetilde{N}] Q_t] \bar{\mathcal{X}}_t +Q_t[-2\widetilde{T}+a(t)^2\widetilde{N}]\Psi_t - a(t)Q_t \widetilde{N} Q_t R_t\\
    &\qquad- [\widetilde{O}+Q_t(\widetilde{M}+ \widetilde{N}a(t))] P^\mathcal{X}_t  + a(t)Q_t\tilde{N}\tilde{F}P^B_t+ Q_t\left\{\left[a(t)^2 \widetilde{N}-2 \widetilde{T}\right]\varphi_t -  \widetilde{N}a(t)\widetilde{F} \right\}\\
    & \qquad - Q_t\left\{\left[a(t)^2 \widetilde{N}+2 \widetilde{T}\right]\varphi_t -  2(A^0-\theta)a(t)\widetilde{N}\widetilde{F} \right\}+ a(t)\begin{pmatrix}
        aA\\0
    \end{pmatrix}\bigg\} dt  + Z^P_t dW^0_t,\\
    P^\mathcal{X}_T & = 2H^\top \bar \X_T + 2B_Th^1+ h_0,\\
    P^B_t &= Z^B_T dW^0_t,\\
    P^B_T &= 2B_T + 2 (h^1)^\top \bar{\mathcal{X}}_T. 
\end{align*}
To simplify the next computations we adopt the following notation: $\mathfrak{X}_t := (\bar{\mathcal{X}}_t, B_t, R_t)^\top$ and $\mathfrak{Y}_t:= (\Psi_t, P^\mathcal{X}_t, P^B_t)^\top$, whose dynamics can be defined as:
\begin{equation*}
    \begin{cases}
        d\mathfrak{X}_t = [\mathfrak{A}^{\mathfrak{X}}_t\mathfrak{X}_t + \mathfrak{A}^{\mathfrak{Y}}_t\mathfrak{Y}_t + \mathfrak{a}_t ] dt + \mathfrak{s} dW^0_t,\\
        d\mathfrak{Y}_t = [\mathfrak{B}^{\mathfrak{X}}_t\mathfrak{X}_t + \mathfrak{B}^{\mathfrak{Y}}_t\mathfrak{Y}_t + \mathfrak{b}_t ] dt + \mathfrak{Z}^0_tdW^0_t,\\
    \end{cases}
\end{equation*}
The coefficients of such FBSDE system are defined as follows:
\begin{align*}
    \mathfrak{A}^{\mathfrak{X}}_t   &= 
    \begin{pmatrix}
        [\tilde{O} + (\tilde{M} + a(t)\tilde{N}) Q_t]   & 0 & \tilde{N}Q_t\\
        - a(t)\tilde{F}^\top\tilde{N} Q_t & 0 & - \tilde{F}^\top\tilde{N} Q_t \\ 
        (- 2\tilde{T} +a(t)^2 \tilde{N})Q_t  &0&[\tilde{O} + (\tilde{M} + a(t) \tilde{N})Q_t] 
    \end{pmatrix}\\
    \mathfrak{A}^{\mathfrak{Y}}_t   &=
    \begin{pmatrix}
        (\tilde{M} + a(t) \tilde{N})  & -\tilde{N} & \tilde{N}\tilde{F}\\
       - a(t)\tilde{F}^\top\tilde{N} & \tilde{F}^\top\tilde{N} & -\tilde{N}_{2,2}\\
        [-2\tilde{T} + \tilde{N}a(t)^2] &-(\widetilde{M} + \widetilde{N}a(t)) & -a(t)\tilde{N}\tilde{F}
    \end{pmatrix}\\
    \mathfrak{a}_t   &=
    \begin{pmatrix}
        (\tilde{M} + a(t) \tilde{N})\varphi_t + a_\mathcal{X}  - 2(A^0-\theta)\tilde{N}\tilde{F}\\
        - a(t) \tilde{F}^\top\tilde{N} \varphi_t +2(A^0-\theta)\tilde{N}_{2,2}\\
        (- 2\tilde{T} + a(t)^2\tilde{N})\varphi_t- 2(A^0-\theta)a(t) \tilde{N}\tilde{F}
    \end{pmatrix}\\
    \mathfrak{B}^{\mathfrak{X}}_t   &=
    \begin{pmatrix}
        - a(t) Q_t \widetilde{N} Q_t    & 0 & -  Q_t \widetilde{N} Q_t\\
     -2 \widetilde{U} +Q_t [-2\widetilde{T}+a(t)^2\widetilde{N}] Q_t  & 0 & - a(t)Q_t \widetilde{N} Q_t\\
     0 &0&0
    \end{pmatrix}\\
    \mathfrak{B}^{\mathfrak{Y}}_t   &=
    \begin{pmatrix}
      - [Q_t ( \widetilde{M} + a(t) \widetilde{N}) +\widetilde{O}] &  Q_t \widetilde{N} &- a(t)\tilde{N}\tilde{F}  \\
         Q_t[-2\widetilde{T}+a(t)^2\widetilde{N}] & - [\widetilde{O}+Q_t(\widetilde{M}+ \widetilde{N}a(t))] & a(t)Q_t\tilde{N}\tilde{F}\\
         0 & 0 & 0 
    \end{pmatrix}\\
    \mathfrak{b}_t   &=
    \begin{pmatrix}
        - Q_t \widetilde{N}\left[a(t) \varphi_t + 2(A^0-\theta) \widetilde{F} \right]\\
        - Q_t\left\{\left[a(t)^2 \widetilde{N} + 2 \widetilde{T}\right]\varphi_t - 2(A^0-\theta) a(t) \widetilde{N} \widetilde{F} \right\} + a(t)
        \begin{pmatrix}
            aA\\0
        \end{pmatrix}\\
        0
    \end{pmatrix}
\end{align*}
Let us also recall the initial and terminal conditions:
\begin{align*}
    \mathfrak{X}_0 &:= (\bar{\mathcal{X}}_0, B_0, R_0)^\top = (X_0, K_0, 0, 0, 0)^\top,\\
    \mathfrak{Y}_T &:= (\Psi_T, P^\mathcal{X}_T, P^B_T)^\top, 
\end{align*}
where
\begin{align*}
    \Psi_T &= (0, 0)^\top,\\ 
    P^\mathcal{X}_T &= 2H^\top \bar{\mathcal{X}}_T + 2B_T h^1 + h^0,\\
    P^B_T &= 2B_T + 2 (h^1)^\top \bar{\mathcal{X}}_T,
\end{align*}
or equivalently in matrix form,
\begin{equation*}
    \mathfrak{Y}_T = 2\begin{pmatrix}
        0 & 0 & 0\\
        H^\top & h^1 & 0\\ 
        (h^1)^\top & 1 & 0
    \end{pmatrix}
    \mathfrak{X}_T
    +
    \begin{pmatrix}
        0  \\
        h^0\\
        0
    \end{pmatrix}.
\end{equation*}
We apply a Riccati approach, assuming that: 
\begin{equation}
\label{eqn: ansatz mathfrak}
\mathfrak{Y}_t := \mathfrak{Q}_t\mathfrak{X}_t + \mathfrak{q}_t,
\end{equation}
for two suitable deterministic functions $\mathfrak{Q}$ and $\mathfrak{q}$. Apply It\^o formula to this claim. 
\begin{align*}
    d\mathfrak{Y}_t 
    &= \mathfrak{Q}'_t \mathfrak{X}_tdt + \mathfrak{Q}_t \{[\mathfrak{A}^{\mathfrak{X}}_t\mathfrak{X}_t + \mathfrak{A}^{\mathfrak{Y}}_t\mathfrak{Y}_t + \mathfrak{a}_t ] dt + \mathfrak{s} dW^0_t\} + \mathfrak{q}'_tdt\\
    &= \{[ \mathfrak{Q}'_t+ \mathfrak{Q}_t \mathfrak{A}^{\mathfrak{X}}_t + \mathfrak{Q}_t\mathfrak{A}^{\mathfrak{Y}}_t \mathfrak{Q}_t]\mathfrak{X}_t + \mathfrak{Q}_t \mathfrak{A}^{\mathfrak{Y}}_t\mathfrak{q}_t + \mathfrak{Q}_t\mathfrak{a}_t + \mathfrak{q}'_t\} dt + \mathfrak{s} dW^0_t,
\end{align*}
while, by definition of $\mathfrak{Y}$, we have 
\begin{align*}
d\mathfrak{Y}_t &=[\mathfrak{B}^{\mathfrak{X}}_t\mathfrak{X}_t + \mathfrak{B}^{\mathfrak{Y}}_t\mathfrak{Y}_t + \mathfrak{b}_t ] dt + \mathfrak{Z}^0_tdW^0_t\\
&= [\mathfrak{B}^{\mathfrak{X}}_t\mathfrak{X}_t + \mathfrak{B}^{\mathfrak{Y}}_t(\mathfrak{Q}_t \mathfrak{X}_t + \mathfrak{q}_t) + \mathfrak{b}_t ] dt + \mathfrak{Z}^0_tdW^0_t\\
&= [(\mathfrak{B}^{\mathfrak{X}}_t + \mathfrak{B}^{\mathfrak{Y}}\mathfrak{Q}_t) \mathfrak{X}_t + \mathfrak{B}^{\mathfrak{Y}}_t \mathfrak{q}_t + \mathfrak{b}_t ] dt + \mathfrak{Z}^0_t dW^0_t
\end{align*}
Comparing the terms we obtain the two ODEs:
\begin{align*}
    \mathfrak{Q}'_t &= -(\mathfrak{Q}_t \mathfrak{A}^{\mathfrak{X}}_t + \mathfrak{Q}_t \mathfrak{A}^{\mathfrak{Y}}_t \mathfrak{Q}_t) + \mathfrak{B}^{\mathfrak{X}}_t +\mathfrak{B}^{\mathfrak{Y}}\mathfrak{Q}_t\\
    \mathfrak{q}'_t &= - ( \mathfrak{Q}_t \mathfrak{a}_t + \mathfrak{Q}_t \mathfrak{A}^{\mathfrak{Y}}_t \mathfrak{q}_t ) + \mathfrak{B}^{\mathfrak{Y}}_t \mathfrak{q}_t +\mathfrak{b}_t.
\end{align*}
with terminal conditions given by
\begin{equation*}
\mathfrak{Q}_T
= 2\begin{pmatrix}
    0 & 0 & 0\\
    H^\top & h^1 & 0\\ 
    (h^1)^\top & 1 &0
\end{pmatrix}, \quad \mathfrak{q}_T 
= \begin{pmatrix}
    0  \\
    h^0\\
    0 \\
\end{pmatrix}.
\end{equation*}

\section{Numerical Experiments}\label{section: numerical_experiments}

This section is devoted to displaying of the power of our results via some numerical illustrations. 
In particular, we compare the equilibrium outcomes of the economy under the regulator's optimal permit-issuance policy with the laissez-faire benchmark, i.e.\ the economy without any carbon-reduction policy (hereafter \emph{no-policy economy}). 
Throughout, we fix a single stochastic scenario (common seed) so that all differences between the two economies are attributable solely to the regulatory intervention and not to sampling variability. The model parameters are summarised in Table~\ref{tab:params}. 

{\footnotesize{
\begin{table}[htb!]
\centering
\caption{\footnotesize{Model parameters used in the numerical experiments.}}
\label{tab:params}
\begin{tabular}{llll}
\toprule
Parameter & Value & Parameter & Value \\
\midrule
$T$            & $5$               & $\sigma_K$    & $0.05$                        \\
$K_0$          & $30$              & $\sigma_E$    & $0.02$                        \\
$E_0$          & $6$               & $\rho$        & $0.92$                        \\
$A_0$          & $10$              & $a$           & $50$                          \\
$a^f$          & $7$               & $b$           & $0.03$                        \\
$a^g$          & $3.5$             & $A$           & $1.7$                         \\
$a^e$          & $2$               & $\gamma$      & $0.99$                        \\
$\delta$       & $0.01$            & $\lambda$     & $9.5\times10^{-3}$            \\
$c_{1,f}$      & $0.01$            & $c_{2,f}$     & $3$                           \\
$c_{1,g}$      & $0.02$            & $c_{2,g}$     & $4$                           \\
$c_{1,e}$      & $80$              & $c_{2,e}$     & $(2\times 0.211)^{-1}$        \\
$c_{2,x}$      & $(2\times285.713)^{-1}$ & $a(t)$ & $1/40$                   \\
$b(t)$         & $1 + e^{t/2}$     & ratio         & $0.5$                         \\
\bottomrule
\end{tabular}
\end{table}
}}

For their choice, when possible, we refer to~\cite{aid2023optimal} and~\cite{del2024mean}. The time horizon is
$T = 5$ years, broadly consistent with a single phase of the EU~ETS. Capital
parameters are chosen so that the no-policy economy reaches a realistic steady state,
while emission parameters are calibrated to reflect the high correlation of sectoral
emissions with common macroeconomic shocks ($\rho = 0.92$), as documented in the
empirical literature \cite{aid2023optimal}. The regulator targets a $50\%$ reduction of
cumulative emissions relative to the no-policy benchmark, i.e.\ $\theta = 0.5\,
\bar{E}^{\mathrm{WP}}_T$, with penalty parameter $\lambda = 9.5 \times 10^{-3}$ and
weight function $a(t) = 1/40$.

\subsection*{Emission reduction and capital preservation}
\label{subsec:emissions}

{\footnotesize{\begin{table}[htb!]
\centering
\caption{\footnotesize{Key outcomes under the optimal regulatory policy vs the no-policy benchmark.}}
\label{tab:results}
\begin{tabular}{lcc}
\toprule
Quantity & No-policy economy & With optimal policy \\
\midrule
Cumulative emissions $E_T$      & $327.6$  & $\approx 163.8$ \\
Emission ratio $E_T / E^{WP}_T$ & $1.000$  & $\approx 0.500$ \\
Capital $K_T$                   & $928.6$  & $\approx 847.2$ \\
Capital ratio $K_T / K^{WP}_T$  & $1.000$  & $\approx 0.913$ \\
\bottomrule
\end{tabular}
\end{table}}}

\begin{figure}[htb!]
    \centering
    \includegraphics[width=0.48\textwidth]{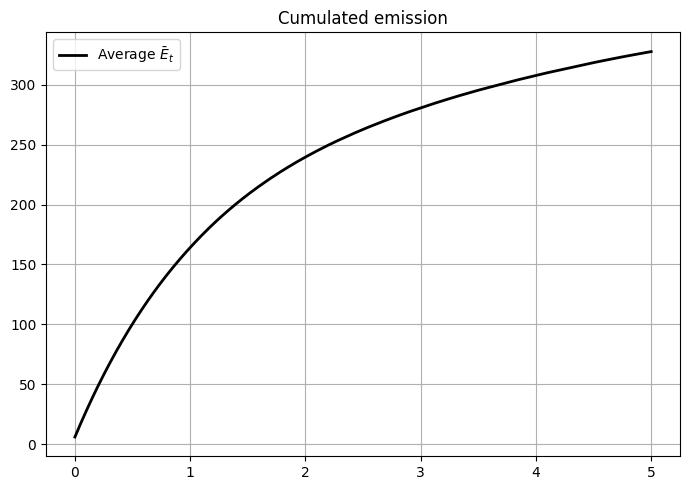}
    \hfill
    \includegraphics[width=0.48\textwidth]{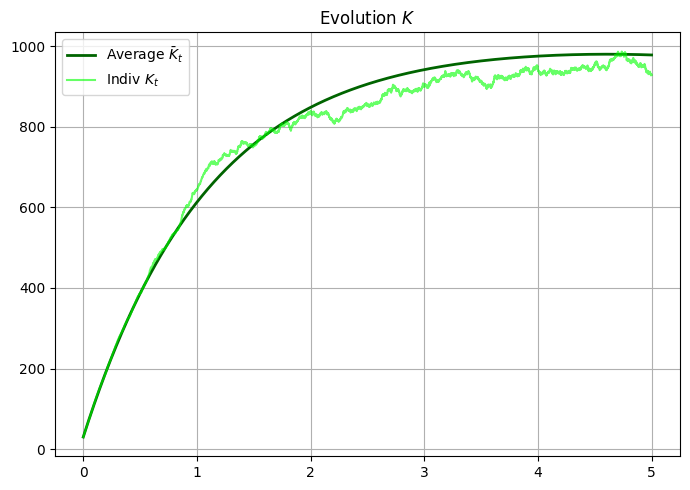}
    \hfill
    \includegraphics[width=0.48\textwidth]{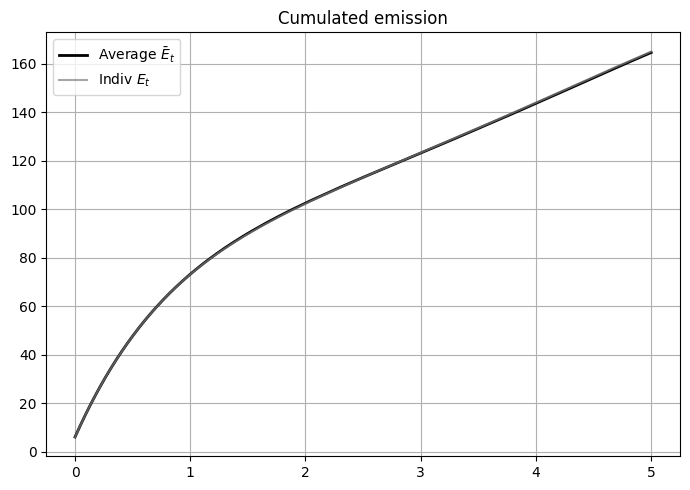}
    \hfill
    \includegraphics[width=0.48\textwidth]{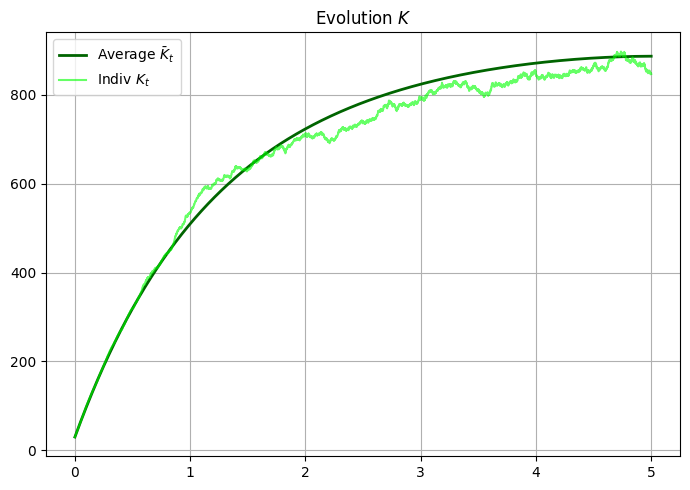}
    \caption{Left: cumulative emissions $E_t$ under the optimal policy (top row) and the
    no-policy economy (bottom row). Right: capital $K_t$ and its mean-field average
    $\bar{K}_t$ under the two regimes. All trajectories share the same Brownian
    realisation.}
    \label{fig:emissions}
\end{figure}

The first key result concerns the effectiveness of the optimal policy in achieving the emission target while preserving the productive capacity of the economy.
Table~\ref{tab:results} summarises the main quantitative outcomes of the experiment, while Figure~\ref{fig:emissions} reports the trajectories of cumulative emissions $E_t$ (left-hand side) and capital $K_t$ (right-hand side) for both economies.
Under the optimal regulatory policy, cumulative emissions at $T=5$ amount to approximately $50\%$ of those recorded in the no-policy economy, confirming that the regulator successfully steers the economy to the prescribed target $\theta$. 
At the same time, the capital stock at $T$ is reduced by only approximately $9\%$ relative to
the no-policy benchmark. 
This result highlights a key feature of the mean-field mechanism: by acting on the equilibrium permit price rather than directly constraining
individual firms, the regulator achieves a large environmental gain at a comparatively modest cost in terms of aggregate productive capacity.

\subsection*{Equilibrium permit price and market clearing}
\label{subsec:price}

\begin{figure}[ht]
    \centering
    \includegraphics[width=0.48\textwidth]{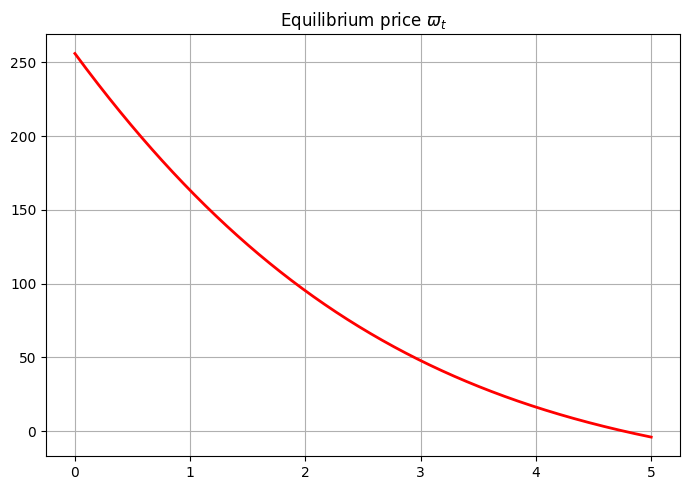}
    \hfill
    \includegraphics[width=0.48\textwidth]{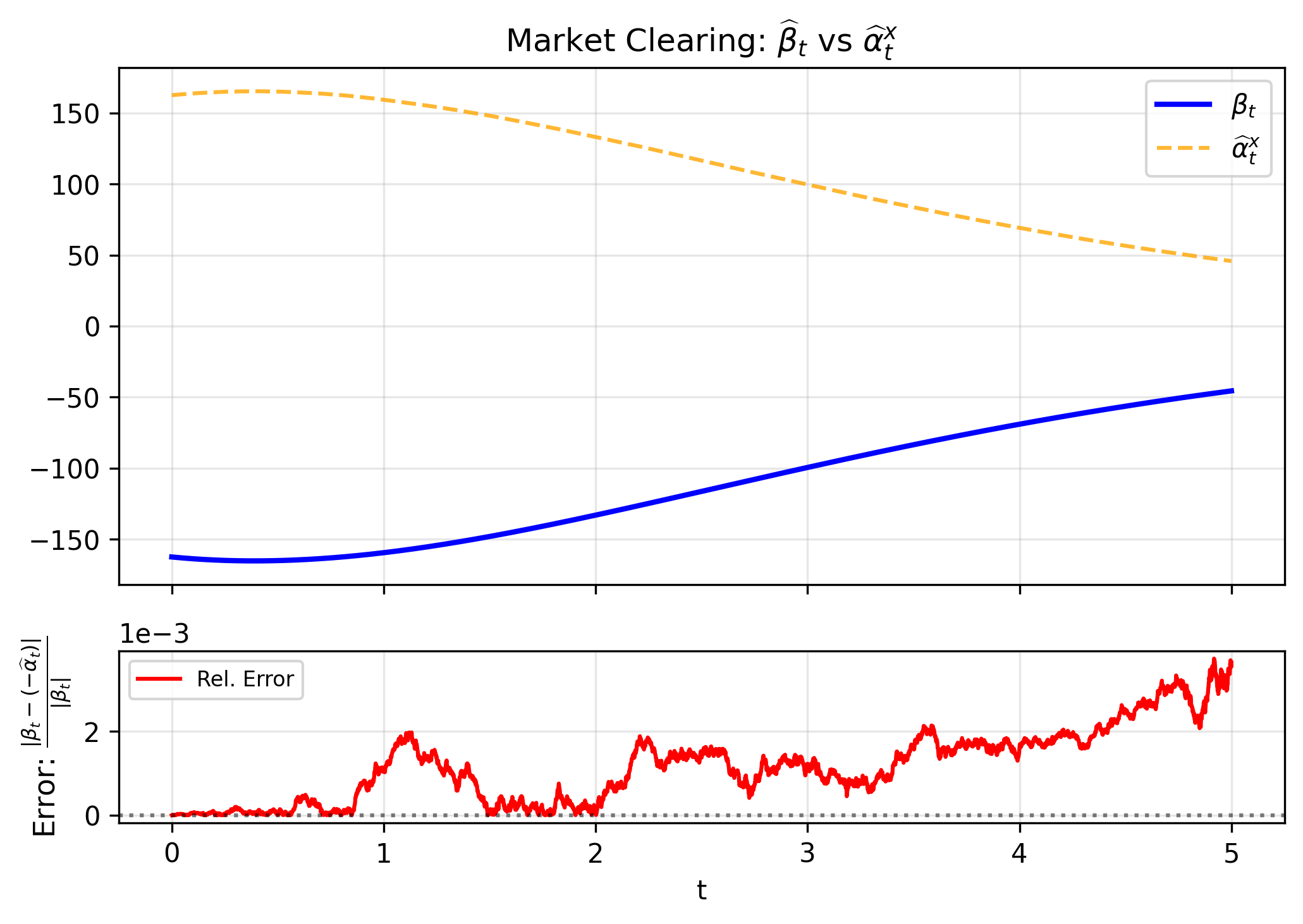}
    \caption{Left: equilibrium permit price $\varpi_t$. Right Top: Comparison between optimal issuance rate $\hat\beta_t$ of the regulator and the ETS trading rate for the representative company multiplied by. Right Down: Relative error between $\hat\beta$ and $-\hat{\alpha}^x$.}
    \label{fig:price_beta}
\end{figure}

Figure~\ref{fig:price_beta} displays the equilibrium permit price $\varpi_t$ and the regulator's optimal issuance rate $\hat\beta_t$.
The permit price starts high and decreases monotonically over the horizon, reflecting the progressive tightening of the cap as cumulative issued permits $B_t$ accumulate.
The issuance rate $\hat\beta_t$ is strictly negative throughout, confirming that the regulator acts as a net supplier of permits at all times, consistently with Remark~\ref{rmk:beta-sign}. Furthermore, we verified numerically that the market-clearing condition $\hat\beta_t \approx\mathbb{E}\bigl[\hat\alpha^x_t \mid \mathcal{F}^0_t\bigr]$, $t\in[0,T]$, is satisfied to numerical precision at each time step, with $ \frac{|\hat\beta_t + \alpha^x_t|}{|\hat\beta_t|} <  0.4\%$ uniformly in $t$.

\subsection*{Energy-mix switching under regulation}
\label{subsec:switch}

\begin{figure}[ht]
    \centering
    \includegraphics[width=0.48\textwidth]{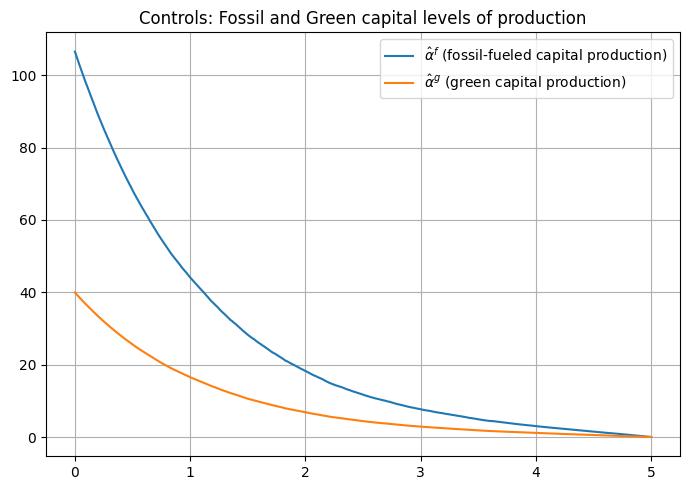}
    \hfill
    \includegraphics[width=0.48\textwidth]{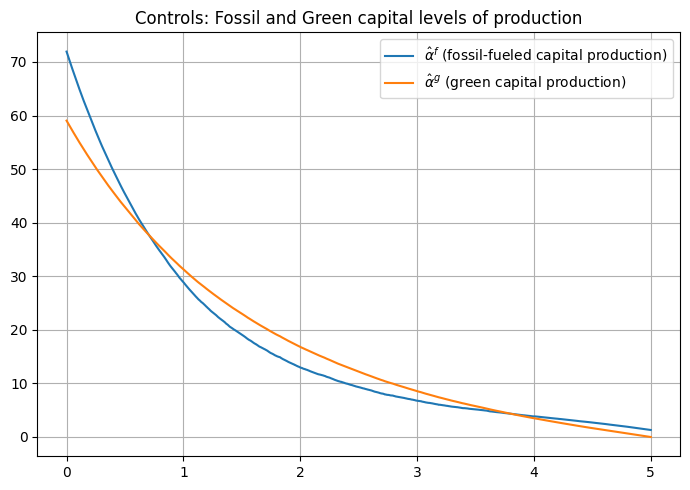}
    \caption{Optimal production controls $\hat\alpha^f_t$ (blue) and $\hat\alpha^g_t$
    (orange) in the no-policy economy (left) and under optimal regulation (right).}
    \label{fig:controls}
\end{figure}

Perhaps the most economically informative finding concerns the firms' production strategy. Figure~\ref{fig:controls} compares the optimal investment rates in fossil-fuel-based capital $\hat\alpha^f_t$ and green capital $\hat\alpha^g_t$ under the two regimes.
In the no-policy economy (left-hand side Figure~\ref{fig:controls}), firms consistently
invest more in fossil-fuel-based capital than in green capital, i.e. $\hat\alpha^f_t > \hat\alpha^g_t$ for all $t\in[0,T]$, as fossil-fuel investment is more productive in the short run.
Under the optimal regulatory policy (right panel), the picture changes markedly. The
rising permit price $\varpi_t$ at the beginning of the horizon makes fossil-fuel production increasingly costly, inducing firms to shift investment towards greener technologies. This results in a \emph{crossing} of the two control trajectories at approximately $t \approx 1$: for $t > 1$, green investment $\hat\alpha^g_t$ dominates fossil-fuel investment $\hat\alpha^f_t$. 

To summarise our findings, the results confirm that the optimal mean-field regulatory policy achieves the desired $50\%$ emission reduction while limiting the loss in aggregate capital to approximately $9\%$. Moreover, the policy endogenously induces a structural shift in the energy mix of the representative firm, with green investment overtaking fossil-fuel investment
roughly one year into the regulation period.

\newpage

\newpage
\appendix

\section{The optimal control problem in absence of carbon reduction policy}
\label{app: unregulated market}

In this Section we consider what the economy described in Section \ref{section: introduction} would be in absence of a policy maker who balances the overall profit of the economy with the emissions of the companies that form the market. 
We consider $N$ companies whose behaviour is described within the filtered probability space $(\Omega, \ItF, \prob, (\F^0, \F^K, \F^E))$, where $(\F^0, \F^K, \F^E)$ are respectively the filtrations generated by the scalar Brownian motion $W^0, W^K, W^E$, in analogy to Section \ref{section: introduction}. The dynamics for the involved state variables remains the same as in Section \ref{sect_model_N}. In particular, the production dynamics is still described by \eqref{eqn: K dynamics}, while the emission is given by \eqref{eqn: E dynamics}. 
What changes with respect to the setting we considered earlier is the firms' objective functional: in absence of ETS-based carbon reduction policies, the target function is defined by
\begin{equation*}
    J^{WP}(v^i, K^i, \mathfrak{m}^{(N)}(\underline{K})) := \E \bigg[ \int_0^T - p(K^i, \mathfrak{m}^{(N)}(\underline{K}_t)AK^i_t + \sum_{p\in\{e,f,g\}} C^p(\alpha^{p,i}_t)dt\bigg], 
\end{equation*}
where $p$ is defined in \eqref{eqn: price good economy n} and $C^p$ is defined in \eqref{eqn: quadratic cost functional}. The control of the $ith$ company is defined by the vector
\begin{equation*}
    v^{WP,i}_t := \begin{pmatrix}
        \alpha^{f,i}_t & \alpha^{g,i}_t & \alpha^{e,i}_t
    \end{pmatrix}, \quad t \in [0,T], \quad i=1,\dots, N.
\end{equation*}
\begin{remark}
\label{rmk: on the measurability of wp problem}
We notice that the optimal control problem is independent of $E$, which is no longer a state variable of the problem. We specify the shape of the running cost which is $f^{WP}:\R^5\to\R$
\begin{equation*}
    f^{WP}(v, K ,\bar{K}):= -[a-b(1-\gamma)AK - b\gamma A\bar{K}]AK + \sum_{p \in \{e,f,g\}}C^p(\alpha^p),
\end{equation*}
where $v=(\alpha^f, \alpha^g, \alpha^e)$. We observe that $\alpha^e$ does not affect the solution, because the emission process $E$ does not appear either in the state variable or in the target function. More in general, the controls can be assumed to be independent of the noises $W^0$ and $W^E$, since these stochastic processes affect neither the dynamics of $K$ nor the target function. This affects the assumptions on the adjoint process determined by the stochastic maximum principle, which is then adapted to $W^K$ only.
\end{remark}

Adopting the approach introduced in Section \ref{section: mean field approach}, we focus on the mean-field optimal control problem defined as the limit of the stochastic game introduced above. The state variable of such an optimal control problem is: 
\begin{equation*}
\begin{cases}
dK_t = (a^f\alpha^f_t + a^g\alpha^g_t - \delta K_t)dt + \sigma^K K_t dW^K_t,\\
K^i_0 = K_0.
\end{cases}
\end{equation*}
The target function is $\bar{J}^{WP}(v, K):= J^{WP}(v, K,\bar{K})$, where $\bar{K}_t:= \E[K_t]$ for all $t\in[0,T]$.

The Hamiltonian function of this problem is defined by 
\begin{align*}
    H(v, K, V, Z, \xi) &:= V(a^f\alpha^f + a^g \alpha^g - \delta K) + f^{WP}(K, v, \xi) + \sigma^KK Z\\
    &=V(a^f\alpha^f+a^g\alpha^g - \delta K)  -[a-b(1-\gamma)AK - b\gamma A\xi]AK \\
    &\qquad+ \sum_{p \in \{e,f,g\}}C^p(\alpha^p)+\sigma^KKZ.
\end{align*}
The candidate optimal control has then the following shape
\begin{equation*}
    \begin{cases}
        \hat\alpha^f_t &= -\frac{1}{2c_{2,f}}(V_ta^f + c_{1,f}),\\
        \hat\alpha^g_t &= -\frac{1}{2c_{2,g}}(V_ta^g + c_{1,g}),\\
        \hat\alpha^e_t &= -\frac{c_{1,e}}{2c_{2,e}},
    \end{cases}
\end{equation*}
where $V$ satisfies
\begin{align*}
    dV_t &= -[-\delta V_t + \sigma^KZ_t - (a-b(1-\gamma)AK_t - b\gamma\xi_t)A - b(1-\gamma)A^2K_t]dt + Z_tdW^K_t\\
    &=[\delta V_t + 2b(1-\gamma)A^2K_t - \sigma^KZ_t + (a-b\gamma A\xi_t)A]dt + Z_tdW^K_t
\end{align*}
Since the candidate optimal control is unique, the optimal state variable is given by 
\begin{equation*}
dK_t = -\bigg[ 
\bigg(\frac{(a^f)^2}{2c_{2f}}+\frac{(a^g)^2}{2c_{2g}}\bigg)V_t + \delta K_t + \frac{a^fc_{1f}}{2c_{2f}}+\frac{a^gc_{1g}}{2c_{2g}}\bigg]dt + \sigma^KK_t dW^K_t.
\end{equation*}
We adopt the following notation:
\begin{equation*}
\begin{cases} 
    A^{WP}_K := -\delta, \quad A^{WP}_V := -\bigg(\frac{(a^f)^2}{2c_{2f}}+\frac{(a^g)^2}{2c_{2g}}\bigg), \quad a^{WP} := -\bigg( \frac{a^fc_{1f}}{2c_{2f}}+\frac{a^gc_{1g}}{2c_{2g}}\bigg),\\
    B^{WP}_K := 2b(1-\gamma)A^2,\quad \bar{B}^{WP}_K:= -b\gamma A^2,\quad B^{WP}_V:= \delta, \quad B^{WP}_Z := - \sigma^K, \quad b^{WP} := aA,
\end{cases}
\end{equation*}
so that the system becomes
\begin{equation*}
    \begin{cases}
        dK_t = (A^{WP}_K K_t + A^{WP}_V V_t + a^{WP}) dt + \sigma^KK_t dW^K_t,\quad K_0=K_0\\
        dV_t = (B^{WP}_K K_t + \bar{B}^{WP}_K \bar{K}_t + B^{WP}_V V_t + B^{WP}_Z Z_t + b^{WP}) dt + Z_t dW^K_t, \quad V_T=0\\
        d\bar{K}_t = (A^{WP}_K \bar{K}_t + A^{WP}_V \bar{V}_t + a^{WP}) dt,\quad K_0=K_0\\
        d\bar{V}_t = (B^{WP}_K \bar{K}_t + \bar{B}^{WP}_K \bar{K}_t + B^{WP}_V \bar{V}_t + B^{WP}_Z P^{WP}_t \sigma^K \bar{K}_t + b^{WP}) dt, \quad V_T=0\\
    \end{cases}
\end{equation*}
We introduce the following ansatz:
\begin{equation*}
    V_t = P^{WP}_t (K_t - \bar{K}_t) + Q^{WP}_t \bar{K}_t + \varphi^{WP}_t, \quad t\in[0,T]
\end{equation*}
for suitable deterministic functions $P^{WP}$, $Q^{WP}$, $\varphi^{WP}$. Applying It\^o's formula to $V$ and adopting the notation $\bar{V}_t := \E [V_t ]$ for $t \in[0,T]$, we get
\begin{align*}
    dV_t &= [\dot{P}^{WP}_t (K_t - \bar{K}_t) + \dot{Q}^{WP}_t \bar{K}_t + \dot{\varphi}^{WP}_t ]dt + P^{WP}_t d(K_t - \bar{K}_t) + Q^{WP}_t d\bar{K}_t\\
    &= \{ (\dot{P}^{WP}_t + P^{WP}_t A^{WP}_K + P^{WP}_tA^{WP}_V P^{WP}_t)K_t\\
    &\qquad + (-\dot{P}^{WP}_t + \dot{Q}^{WP}_t - P^{WP}_t A^{WP}_K + Q^{WP}_t A^{WP}_K - P^{WP}_t A^{WP}_V P^{WP}_t + Q^{WP}_tA^{WP}_VQ^{WP}_t) \bar{K}_t\\
    &\qquad +(Q^{WP}_t A^{WP}_V \varphi^{WP}_t + \dot{\varphi}^{WP}_t+ Q^{WP}_t a^{WP})\} dt + P^{WP}_t \sigma^K K_t dW^K_t
\end{align*}
We notice that, comparing the volatility terms:
\begin{equation*}
    Z_t = P^{WP}_t \sigma^KK_t,\quad t\in[0,T].
\end{equation*}
On the other hand, we have that:
\begin{align*}
    dV_t &= (B^{WP}_K K_t + \bar{B}^{WP}_K \bar{K}_t + B^{WP}_V V_t + B^{WP}_Z Z_t + b^{WP}) dt + Z_t dW^K_t\\
    &= [(B^{WP}_K + B^{WP}_VP^{WP}_t + B^{WP}_Z P^{WP}_t \sigma^K)K_t+(\bar{B}^{WP}_K - B^{WP}_V P^{WP}_t + B^{WP}_VQ^{WP}_t) \bar{K}_t\\
    &\qquad+ b^{WP} + B^{WP}_V\varphi^{WP}_t]dt + (P^{WP}_t \sigma^KK_t)  dW^K_t,
\end{align*}
which leads to: 
\begin{equation*}
    \begin{cases}
        \dot{P}^{WP}_t + P^{WP}_t A^{WP}_K + P^{WP}_tA^{WP}_V P^{WP}_t = B^{WP}_K + B^{WP}_VP^{WP}_t + B^{WP}_Z P^{WP}_t \sigma^K\\
        -\dot{P}^{WP}_t + \dot{Q}^{WP}_t - P^{WP}_t A^{WP}_K + Q^{WP}_t A^{WP}_K - P^{WP}_t A^{WP}_V P^{WP}_t + Q^{WP}_tA^{WP}_VQ^{WP}_t\\
        \qquad=\bar{B}^{WP}_K - B^{WP}_V P^{WP}_t + B^{WP}_VQ^{WP}_t\\
        Q^{WP}_t A^{WP}_V \varphi^{WP}_t + \dot{\varphi}^{WP}_t+ Q^{WP}_t a^{WP}= b^{WP}+ B^{WP}_V\varphi^{WP}_t.
    \end{cases}
\end{equation*}
In particular, the system can be rearranged as follows
\begin{equation*}
    \begin{cases}
        \dot{P}^{WP}_t &=  B^{WP}_K + (-A^{WP}_K + B^{WP}_V + B^{WP}_Z \sigma^K) P^{WP}_t -A^{WP}_V (P^{WP}_t)^2\\
        \dot{Q}^{WP}_t &= (\bar{B}^{WP}_K + B^{WP}_K) + B^{WP}_Z \sigma^K P^{WP}_t  + (B^{WP}_V - A^{WP}_K)Q^{WP}_t \\
        &\qquad - A^{WP}_V(Q^{WP}_t)^2\\
        \dot{\varphi}^{WP}_t &= b^{WP} - a^{WP}Q^{WP}_t + (B^{WP}_V - Q^{WP}_t A^{WP}_V) \varphi^{WP}_t
    \end{cases}
\end{equation*}
and the terminal conditions are
\begin{equation*}
    P^{WP}_T = Q^{WP}_T = \varphi^{WP}_T = 0.
\end{equation*}

\section{Short Descriptions of the Implementation}
\label{app: sbs implementation} 
Here we summarise how one can proceed to simulate the model in practice. 
This is displayed with the purpose to show how simple it is to obtain a solution to the regulator's problem  under our modelling assumption.


\paragraph{\textbf{Step 1.}}

We start by solving the following ODEs. Let us notice that the first two equations are matrix Riccati equations.
\begin{align}
    &\label{eqn: P} {P}_t' +{P}_t A_{\mathcal{Y}} {P}_t + {P}_t A_{\mathcal{X}}  + A_{\mathcal{X}} {P}_t  + C_{\mathcal{X}} {P}_t C_{\mathcal{X}} - B_{\mathcal{X}} = 0,
    \\
    & {Q}_t' + Q_t \hat{A}_{\mathcal{Y}} Q_t + Q_t A_{\mathcal{X}} +   A_{\mathcal{X}} Q_t + C_{\mathcal{X}} {P}_t C_{\mathcal{X}} - B_{\mathcal{X}} - \overline{B}_{\mathcal{X}}   = 0,
    \\
    &\label{eqn: odevarphi} {\varphi}_t'  + [Q_t \hat{A}_{\mathcal{Y}} + A_{\mathcal{X}}]\varphi_t + Q_t a_\mathcal{X} - b_\mathcal{Y} =0.
\end{align}
with terminal condition $P_T=Q_T=Q_4$ and $\varphi_T=0$.

The coefficients of these equations are the following.
\begin{alignat*}{3}
    A_{\mathcal{X}} &= \begin{pmatrix}
        -\delta & 0 \\ 0 & 0
    \end{pmatrix}, &\
    A_{\mathcal{Y}} &= -\frac{1}{2}\begin{pmatrix}
        \frac{(a^f)^2}{c_{2,f}} + \frac{(a^g)^2}{c_{2,g}} & -\frac{a^e a^f}{c_{2,f}} \\
        -\frac{a^e a^f}{c_{2,f}} & \frac{(a^e)^2}{c_{2,f}} + \frac{1}{c_{2,e}} + \frac{1}{c_{2,x}}
    \end{pmatrix}, &\
    a_{\mathcal{X}} &= -\frac{1}{2}\begin{pmatrix}
        \frac{a^f c_{1,f}}{c_{2,f}} + \frac{a^g c_{1,g}}{c_{2,g}} \\
        -\frac{a^e c_{1,f}}{c_{2,f}} + \frac{c_{1,e}}{c_{2,e}}
    \end{pmatrix},
     \\[10pt]
    \overline{A}_{\mathcal{Y}} &= \begin{pmatrix}
        0 & 0 \\ 0 & \frac{1}{2c_{2,x}}
    \end{pmatrix}, &\
    \hat{A}_{\mathcal{Y}} &= -\frac{1}{2}\begin{pmatrix}
        \frac{(a^f)^2}{c_{2,f}} + \frac{(a^g)^2}{c_{2,g}} & -\frac{a^e a^f}{c_{2,f}} \\
        -\frac{a^e a^f}{c_{2,f}} & \frac{(a^e)^2}{c_{2,f}} + \frac{1}{c_{2,e}}
    \end{pmatrix},
    &\qquad
    \widetilde{F} &= \begin{pmatrix} 0 \\ 1 \end{pmatrix}, \\[10pt]
    \overline{B}_{\mathcal{X}} &= \begin{pmatrix}
        b\gamma A^2 & 0 \\ 0 & 0
    \end{pmatrix}, &\
    B_{\mathcal{X}} &= \begin{pmatrix}
        2b(1-\gamma)A^2 & 0 \\ 0 & 0
    \end{pmatrix}, &\
    C_{\mathcal{X}} &= \begin{pmatrix}
        \sigma^K & 0 \\ 0 & 0
    \end{pmatrix}, \\[10pt]
    c_0 &= \begin{pmatrix} 0 \\ -\sigma^E\sqrt{1-\rho^2} \end{pmatrix}, &\
    b_{\mathcal{Y}} &= \begin{pmatrix} aA \\ 0 \end{pmatrix}, &\
    Q_4 &= \begin{pmatrix} 0 & 0 \\ 0 & \lambda \end{pmatrix}, \\[10pt]
    \widetilde{N} &= \begin{pmatrix}
        0 & 0 \\ 0 & \frac{1}{2(b(t)-c_{2,x})}
    \end{pmatrix}.
\end{alignat*}


\paragraph{\textbf{Step 2.}}

Exploiting the coefficients above together with the solutions to the system of ODEs defined above we can define:
\begin{align*}
    \mathfrak{A}^{\mathfrak{X}}_t   &= 
    \begin{pmatrix}
        [\tilde{O} + (\tilde{M} + a(t)\tilde{N}) Q_t]   & 0 & \tilde{N}Q_t\\
        - a(t)\tilde{N}\tilde{F}^\top Q_t & 0 & - \tilde{N}\tilde{F}^\top Q_t \\ 
        (- 2\tilde{T} +a(t)^2 \tilde{N})Q_t  &0&[\tilde{O} + (\tilde{M} + a(t) \tilde{N})]Q_t 
    \end{pmatrix}\\
    \mathfrak{A}^{\mathfrak{Y}}_t   &=
    \begin{pmatrix}
        (\tilde{M} + a(t) \tilde{N})  & -\tilde{N} & \tilde{N}\tilde{F}\\
       - a(t)\tilde{N}\tilde{F}^\top & \tilde{N}\tilde{F}^\top & -\tilde{N}_{2,2}\\
        [-2\tilde{T} + \tilde{N}a(t)^2] &-(\widetilde{M} + \widetilde{N}a(t)) & -a(t)\tilde{N}\tilde{F}
    \end{pmatrix}\\
    \mathfrak{a}_t   &=
    \begin{pmatrix}
        (\tilde{M} + a(t) \tilde{N})\varphi_t + a_\mathcal{X}  - 2(A^0-\theta)\tilde{N}\tilde{F}\\
        - a(t) \tilde{N}\tilde{F}^\top \varphi_t +2(A^0-\theta)\tilde{N}_{2,2}\\
        (- 2\tilde{T} + a(t)^2\tilde{N})\varphi_t- 2(A^0-\theta)a(t) \tilde{N}\tilde{F}
    \end{pmatrix}\\
    \mathfrak{B}^{\mathfrak{X}}_t   &=
    \begin{pmatrix}
        - a(t) Q_t \widetilde{N} Q_t    & 0 & -  Q_t \widetilde{N} Q_t\\
     -2 \widetilde{U} +Q_t [-2\widetilde{T}+a(t)^2\widetilde{N}] Q_t \bar{\mathcal{X}}_t & 0 & - a(t)Q_t \widetilde{N} Q_t\\
     0 &0&0
    \end{pmatrix}\\
    \mathfrak{B}^{\mathfrak{Y}}_t   &=
    \begin{pmatrix}
      - [Q_t ( \widetilde{M} + a(t) \widetilde{N}) +\widetilde{O}] &  Q_t \widetilde{N} &- a(t)\tilde{N}\tilde{F}  \\
         Q_t[-2\widetilde{T}+a(t)^2\widetilde{N}] & - [\widetilde{O}+Q_t(\widetilde{M}+ \widetilde{N}a(t))] & a(t)Q_t\tilde{N}\tilde{F}\\
         0 & 0 & 0 
    \end{pmatrix}\\
    \mathfrak{b}_t   &=
    \begin{pmatrix}
        - Q_t \widetilde{N}\left[a(t) \varphi_t + 2(A^0-\theta) \widetilde{F} \right]\\
        - Q_t\left\{\left[a(t)^2 \widetilde{N} + 2 \widetilde{T}\right]\varphi_t - 2(A^0-\theta) a(t) \widetilde{N} \widetilde{F} \right\} + a(t)
        \begin{pmatrix}
            aA\\0
        \end{pmatrix}\\
        0
    \end{pmatrix}
\end{align*}
These are the coefficients appearing the following symmetric Riccati equation:
\begin{align*}
    \mathfrak{Q}'_t &= -(\mathfrak{Q}_t \mathfrak{A}^{\mathfrak{X}}_t + \mathfrak{Q}_t \mathfrak{A}^{\mathfrak{Y}}_t \mathfrak{Q}_t) + \mathfrak{B}^{\mathfrak{X}}_t +\mathfrak{B}^{\mathfrak{Y}}\mathfrak{Q}_t\\
    \mathfrak{q}'_t &= - ( \mathfrak{Q}_t \mathfrak{a}_t + \mathfrak{Q}_t \mathfrak{A}^{\mathfrak{Y}}_t \mathfrak{q}_t ) + \mathfrak{B}^{\mathfrak{Y}}_t \mathfrak{q}_t +\mathfrak{b}_t.
\end{align*}
with terminal conditions given by
\begin{equation*}
\mathfrak{Q}_T
= 2\begin{pmatrix}
    0 & 0 & 0\\
    H^\top & h^1 & 0\\ 
    (h^1)^\top & 1 &0
\end{pmatrix}, \quad \mathfrak{q}_T 
= \begin{pmatrix}
    0  \\
    h^0\\
    0 \\
\end{pmatrix}.
\end{equation*}


\paragraph{\textbf{Step 3.}}
The state variable of the problem is then defined by $\mathfrak{X}=(\bar K, \bar X, R)$, which solves the following SDE:
\begin{equation*}
    \begin{cases}
        d\mathfrak{X}_t = [\mathfrak{A}^{\mathfrak{X}}_t\mathfrak{X}_t + \mathfrak{A}^{\mathfrak{Y}}_t\mathfrak{Y}_t + \mathfrak{a}_t ] dt + \mathfrak{s} dW^0_t,\\
        \mathfrak{X}_0 = ((K^0, X^0), 0, (0, 0)),\\
    \end{cases}
\end{equation*}
where we recall that $\mathfrak{Y}_t =\mathfrak{Q}_t\mathfrak{X}_t + \mathfrak{q}_t$.

The regulator's optimal control can actually be rewritten just in terms of $\mathfrak{X}_t$:
\begin{align*}
    \hat\beta_t 
    := \hat\beta(t;\mathfrak{X}_t)&= -
    \frac{1}{2Q^\beta_t} \bigg( 
    \begin{pmatrix}
        q^{\beta\mathcal{X}}_t & 0 & (B^\beta_t)^\top
    \end{pmatrix}
    \mathfrak{X}_t + 
    \begin{pmatrix}
        q^{\beta\Psi}_t & (A^\beta_t)^\top & 1
    \end{pmatrix}
    \mathfrak{Y}_t + q^\beta_t\bigg)\\
    &= -
    \frac{1}{2Q^\beta_t} \bigg[\bigg( 
    \begin{pmatrix}
        q^{\beta\mathcal{X}}_t & 0 & (B^\beta_t)^\top
    \end{pmatrix}
    + 
    \begin{pmatrix}
        q^{\beta\Psi}_t & (A^\beta_t)^\top & 1
    \end{pmatrix}
    \mathfrak{Q}_t\bigg)\mathfrak{X}_t \\
    &\qquad +
    \begin{pmatrix}
        q^{\beta\Psi}_t & (A^\beta_t)^\top & 1
    \end{pmatrix}
    \mathfrak{q}_t+ q^\beta_t\bigg].
\end{align*}
The equilibrium price, defined in \eqref{eqn: mf eq price}, can be rewritten in term of $\mathfrak{X}$ using the two ansatz \eqref{eq:ANSATZ_Y} and \eqref{eqn: ansatz mathfrak}. To do so, we define 
\begin{equation*}
    \hat{e}^2_2 = 
    \begin{pmatrix}
        0\\
        1
    \end{pmatrix},\quad
    \Pi_{1,2} = 
    \begin{pmatrix}
        1& 0& 0& 0& 0\\
        0& 1& 0& 0& 0
    \end{pmatrix}.
\end{equation*}
We have $\mathcal{X}_t = \Pi_{1,2} \mathfrak{X}_t$ and $\Psi_t = \Pi_{1,2} \mathfrak{Y}_t$ for all $t\in[0,T]$:
\begin{align*}
    \w_t 
    &= -\bar{Y}_t + 2c_{2,x}\hat\beta_t,\\
    &= -(\hat{e}^2_2)^\top [(Q_t\Pi_{1,2} + \Pi_{1,2}\mathfrak{Q}_t)\mathfrak{X}_t + \Pi_{1,2}\mathfrak{q}_t + \varphi_t]  + 2c_{2,x}\hat\beta(t, \mathfrak{X}_t).
\end{align*}
Finally, to characterise the strategy adopted by the representative firm after the regulator has set the permit price, we derive the explicit form of the optimal control as a function of the price process. By \eqref{eqn: optimal control}, the optimal control is given by
\begin{align*}
\widehat{v}_t
= -Q_3^{-1} \bigg( q_2(\varpi_t) + \frac{1}{2} A_2^\top \mathcal{Y}_t \bigg),
\end{align*}
where
\begin{align*}
\mathcal{Y}_t &= P_t(\mathcal{X}_t - \bar{\mathcal{X}}_t) + Q_t \bar{\mathcal{X}}_t + \Psi_t + \varphi_t,\\
q_2(\varpi_t) &= \bigg( \tfrac{1}{2} c_{1,f},\; \tfrac{1}{2} c_{1,g},\; \tfrac{1}{2} c_{1,e},\; \tfrac{1}{2} \varpi_t \bigg)^\top .
\end{align*}
Hence, to simulate the optimal controls of the typical agent, we need its state variable, which is given by: 
\begin{align*}
d\mathcal{X}_t &= \Bigg[ A_\mathcal{X} \mathcal{X}_t + A_{\mathcal{Y}} \bigg( P_t (\mathcal{X}_t - \bar{\mathcal{X}}_t) + Q_t \bar{\mathcal{X}}_t + \Psi_t + \varphi_t\bigg) + \bar{A}_{\mathcal{Y}}\bigg( Q_t \bar{\mathcal{X}}_t + \Psi_t + \varphi_t\bigg)  \\
&\qquad+ A_\beta \hat\beta(t;\mathfrak{X}_t) + a_\mathcal{X}\Bigg] dt+ (C_\mathcal{X} \mathcal{X}_t + c_K)dW^K_t + c_X dW^E_t + c_0 dW^0_t\\ 
&= \Big[ (A_\mathcal{X} + A_\mathcal{Y}P_t) \mathcal{X}_t + \big[(A_\mathcal{Y}(-P_t + Q_t) + \bar{A}_\mathcal{Y}Q_t) \Pi_{1,2} + (A_\mathcal{Y} + \bar{A}_\mathcal{Y})\Pi_{1,2}\mathfrak{Q}_t\big]\mathfrak{X}_t \\
&\qquad + (A_\mathcal{Y} + \bar{A}_\mathcal{Y})\Pi_{1,2} \mathfrak{q}_t+ (A_\mathcal{Y}+ \bar{A}_\mathcal{Y})\varphi_t+ A_\beta \hat\beta(t;\mathfrak{X}_t) + a_\mathcal{X}\Big] dt\\
&\qquad+ (C_\mathcal{X} \mathcal{X}_t + c_K)dW^K_t + c_X dW^E_t + c_0 dW^0_t.
\end{align*}
Once we simulated $\mathfrak{X}$ we can simulate also the trajectory of $\mathcal{X}$.

The optimal controls are then given by:
\begin{align*}
    \hat{v}_t &= -Q_3^{-1} \bigg( q_2(\varpi_t) + \frac{1}{2} A_2^\top \mathcal{Y}_t \bigg)\\
    &= -Q_3^{-1} \bigg\{ q_2(\varpi(t; \mathfrak{X}_t)) + \frac{1}{2} A_2^\top \big[P_t \mathcal{X}_t+ ((-P_t+Q_t)\Pi_{1,2} + \Pi_{1,2}\mathfrak{Q}_t)\mathfrak{X}_t + \Pi_{1,2}\mathfrak{q}_t + \varphi_t\big] \bigg\}.
\end{align*}
Recalling that $\hat{v}_t = (\hat\alpha^f_t, \hat\alpha^g_t, \hat\alpha^e_t, \hat\alpha^x_t)$, we conclude that the emission of the typical agent is
\begin{equation*}
dE_t = (a^e \hat\alpha^{f}_t - \hat\alpha^{e}_t) dt + \sigma^E(\rho dW^{E}_t + \sqrt{1- \rho^2} dW^0_t), \quad E_0 = E^0.
\end{equation*}
More specifically, its expectation satisfies
\begin{equation*}
    \bar{E}_t = \E[E_t \mid \ItF^0_t] = \int_0^t\E[ \hat{\alpha}^{x}_s \mid \ItF^0_s] ds - \bar{X}_t + X^0+ E^0 = -(\hat{e}^5_1 + \hat{e}^5_3)^\top \mathfrak{X}_t + A^0,
\end{equation*}
where $\hat{e}^{j}_i$ is the $i-th$ element of the canonical basis in $\R^j$.

\newpage

\bibliographystyle{alpha}
\bibliography{mfcarbonemission}

@book{liptser2013statistics,
  title={Statistics of random processes: I. General theory},
  author={Liptser, Robert S and Shiryaev, Albert N},
  volume={5},
  year={2013},
  publisher={Springer Science \& Business Media}
}

@book{wang2018introduction,
  title={An introduction to optimal control of FBSDE with incomplete information},
  author={Wang, Guangchen and Wu, Zhen and Xiong, Jie and others},
  year={2018},
  publisher={Springer}
}

@book{carmona2018probabilistic2,
	author = {Carmona, Ren{\'e} and Delarue, Fran{\c{c}}ois},
	edition = {1},
	publisher = {Springer Cham},
	title = {Probabilistic Theory of Mean Field Games with Applications II},
	volume = {II},
	year = {2018}}

@article{aid2023optimal,
	author = {A{\"\i}d, Ren{\'e} and Biagini, Sara},
	journal = {Mathematical Finance},
	number = {1},
	pages = {80--115},
	publisher = {Wiley Online Library},
	title = {Optimal dynamic regulation of carbon emissions market},
	volume = {33},
	year = {2023}}

@article{fujii2022equilibrium,
	author = {Fujii, Masaaki and Takahashi, Akihiko},
	journal = {ESAIM: Control, Optimisation and Calculus of Variations},
	pages = {21},
	publisher = {EDP Sciences},
	title = {Equilibrium price formation with a major player and its mean field limit},
	volume = {28},
	year = {2022}}

@article{carmona2013mean,
author = {Ren{\'e} Carmona and Fran{\c{c}}ois Delarue},
title = {{Mean field forward-backward stochastic differential equations}},
volume = {18},
journal = {Electronic Communications in Probability},
number = {none},
publisher = {Institute of Mathematical Statistics and Bernoulli Society},
pages = {1 -- 15},
year = {2013},
doi = {10.1214/ECP.v18-2446},
URL = {https://doi.org/10.1214/ECP.v18-2446}
}

@article{carmona2015forward,
	author = {Carmona, Ren{\'e} and Delarue, Fran{\c{c}}ois},
	journal = {The Annals of Probability},
	pages = {2647--2700},
	publisher = {JSTOR},
	title = {Forward--backward stochastic differential equations and controlled McKean--Vlasov dynamics},
	year = {2015}}

@article{fujii2022mean,
	author = {Fujii, Masaaki and Takahashi, Akihiko},
	journal = {SIAM Journal on Control and Optimization},
	number = {1},
	pages = {259--279},
	publisher = {SIAM},
	title = {A mean field game approach to equilibrium pricing with market clearing condition},
	volume = {60},
	year = {2022}}

@article{follmer1974random,
	author = {F{\"o}llmer, Hans},
	journal = {Journal of mathematical economics},
	number = {1},
	pages = {51--62},
	publisher = {Elsevier},
	title = {Random economies with many interacting agents},
	volume = {1},
	year = {1974}}

@article{lacker2019mean,
	author = {Lacker, Daniel and Zariphopoulou, Thaleia},
	journal = {Mathematical Finance},
	number = {4},
	pages = {1003--1038},
	publisher = {Wiley Online Library},
	title = {Mean field and n-agent games for optimal investment under relative performance criteria},
	volume = {29},
	year = {2019}}

@article{del2024mean,
	author = {Del Sarto, Gianmarco and Leocata, Marta and Livieri, Giulia},
	journal = {arXiv preprint arXiv:2407.12754},
	title = {A Mean Field Game approach for pollution regulation of competitive firms},
	year = {2024}}

@article{carmona2009optimal,
  title={Optimal stochastic control and carbon price formation},
  author={Carmona, Ren{\'e} and Fehr, Max and Hinz, Juri},
  journal={SIAM Journal on Control and Optimization},
  volume={48},
  number={4},
  pages={2168--2190},
  year={2009},
  publisher={SIAM}
}

@article{carmona2010market,
  title={Market design for emission trading schemes},
  author={Carmona, Ren{\'e} and Fehr, Max and Hinz, Juri and Porchet, Arnaud},
  journal={Siam Review},
  volume={52},
  number={3},
  pages={403--452},
  year={2010},
  publisher={SIAM}
}

@article{kollenberg2019dynamic,
  title={Dynamic supply adjustment and banking under uncertainty in an emission trading scheme: the market stability reserve},
  author={Kollenberg, Sascha and Taschini, Luca},
  journal={European Economic Review},
  volume={118},
  pages={213--226},
  year={2019},
  publisher={Elsevier}
}

@article{chan2017fracking,
  title={Fracking, renewables, and mean field games},
  author={Chan, Patrick and Sircar, Ronnie},
  journal={SIAM Review},
  volume={59},
  number={3},
  pages={588--615},
  year={2017},
  publisher={SIAM}
}

@article{carmona2022mean,
  title={Mean field models to regulate carbon emissions in electricity production},
  author={Carmona, Ren{\'e} and Dayan{\i}kl{\i}, G{\"o}k{\c{c}}e and Lauri{\`e}re, Mathieu},
  journal={Dynamic Games and Applications},
  volume={12},
  number={3},
  pages={897--928},
  year={2022},
  publisher={Springer}
}

@article{dumitrescu2024energy,
  title={Energy transition under scenario uncertainty: a mean-field game of stopping with common noise},
  author={Dumitrescu, Roxana and Leutscher, Marcos and Tankov, Peter},
  journal={Mathematics and Financial Economics},
  volume={18},
  number={2},
  pages={233--274},
  year={2024},
  publisher={Springer}
}

@article{carmona2016,
author = {Ren{\'e} Carmona and Fran{\c{c}}ois Delarue and Daniel Lacker},
title = {{Mean field games with common noise}},
volume = {44},
journal = {The Annals of Probability},
number = {6},
publisher = {Institute of Mathematical Statistics},
pages = {3740 -- 3803},
year = {2016},
doi = {10.1214/15-AOP1060},
URL = {https://doi.org/10.1214/15-AOP1060}
}
\end{document}